%% file: WORS5.tex
\documentclass
{amsproc}
\usepackage{amsmath}
\usepackage{amsthm}
\usepackage{amssymb}
\usepackage{graphicx}
\usepackage{multirow}
\usepackage{caption}
\usepackage{subfig}

\DeclareGraphicsExtensions{.png,.pdf,.eps}

\usepackage[all,2cell]{xy}
\usepackage{tikz}
\usepackage{pgf}
\usepackage{enumerate}
\usetikzlibrary{cd}

\usepackage{array}
\newcolumntype{H}{>{\setbox0=\hbox\bgroup}c<{\egroup}@{}}
\newcolumntype{L}{>{$}l<{$}} 
\newcolumntype{C}{>{$}c<{$}} 

\usepackage{hyperref}

\newtheorem*{theorem*}{Theorem}
\newtheorem*{corollary*}{Corollary}
\newtheorem{theorem}{Theorem}[section]
\newtheorem{lemma}[theorem]{Lemma}
\newtheorem{corollary}[theorem]{Corollary}
\newtheorem{proposition}[theorem]{Proposition}

\theoremstyle{definition}

\newtheorem{definition}[theorem]{Definition}

\newtheorem{setup}[theorem]{Setup}
\newtheorem{problem}{Problem}
\newtheorem{remark}[theorem]{Remark}

\newtheorem{set}[theorem]{Setup}

\title[Rational homogeneous spaces as geometric realizations]{Rational homogeneous spaces as geometric realizations of birational transformations}

\author[Occhetta]{Gianluca Occhetta}
\address{Dipartimento di Matematica, Universit\`a degli Studi di Trento, via
Sommarive 14 I-38123 Povo di Trento (TN), Italy}
 
\email{gianluca.occhetta@unitn.it, eduardo.solaconde@unitn.it}

\author[Romano]{Eleonora A. Romano}
\address{Dipartimento di Matematica, Universit\`a degli Studi di Genova, via Dodecaneso 35, I-16146, Genova (GE), Italy}
\email{eleonoraanna.romano@unige.it}

\author[Sol\'a Conde]{Luis E. Sol\'a Conde}

\author[Wi\'sniewski]{Jaros\l{}aw A. Wi\'sniewski}
\address{Instytut Matematyki UW, Banacha 2, 02-097 Warszawa, Poland}
\email{J.Wisniewski@uw.edu.pl}


\thanks{First and third author supported by PRIN project ``Geometria delle variet\`a algebriche''. Fourth author supported by Polish National Science Center project 2016/23/G/ST1/04282.}

\input{JW-definitions.tex}

\begin{document}
\begin{abstract}
A geometric realization of a birational map $\psi$ among two complex projective varieties is a variety $X$ endowed with a $\C^*$-action inducing $\psi$ as the natural birational map among two extremal geometric quotients. In this paper we study geometric realizations of some classic birational maps --inversion maps, special Cremona transformations, special birational transformations of type $(2,1)$--, by considering $\C^*$-actions on certain rational homogeneous spaces and their subvarieties.  
\end{abstract}
\maketitle
\tableofcontents

\input{intro}

\input{prelim}

\input{examples}

\input{cremona}

\input{RHbundles}

\bibliographystyle{plain}
\bibliography{bibliomin}
\end{document}

%% file: JW-definitions.tex
\usepackage[textsize=tiny]{todonotes}

\usepackage{enumitem}

\newcommand\ignore[1]{}

\DeclareMathOperator{\HH}{H}

\newcommand\CC{{\mathbb{C}}}

\newcommand\ZZ{{\mathbb{Z}}}

\def\C{{\mathbb C}}

\def\H{{\mathbb H}}

\def\Oc{{\mathbb O}}
\def\P{{\mathbb P}}
\def\Q{{\mathbb Q}}
\def\R{{\mathbb R}}
\def\Z{{\mathbb Z}}

\def\cC{{\mathcal C}}
\def\cD{{\mathcal D}}
\def\cE{{\mathcal E}}

\def\cN{{\mathcal{N}}}
\def\cO{{\mathcal{O}}}

\def\cQ{{\mathcal{Q}}}

\def\cS{{\mathcal S}}
\def\cT{{\mathcal T}}

\def\cY{{\mathcal Y}}
\def\Q{{\mathbb{Q}}}

\def\fg{{\mathfrak g}}
\def\fh{{\mathfrak h}}
\def\fp{{\mathfrak p}}
\def\fb{{\mathfrak b}}

\def\operatorname#1{\mathop{\rm #1}\nolimits}

\def\DA{{\rm A}}
\def\DB{{\rm B}}
\def\DC{{\rm C}}
\def\DD{{\rm D}}
\def\DE{{\rm E}}
\def\DF{{\rm F}}
\def\DG{{\rm G}}

\def\Exc{\operatorname{Exc}}

\def\Hom{\operatorname{Hom}}

\def\Pic{\operatorname{Pic}}
\def\Hom{\operatorname{Hom}}

\def\id{\operatorname{id}}

\def\deg{\operatorname{deg}}

\def\conj{\operatorname{conj}}

\def\Nef{{\operatorname{Nef}}}

\def\NU{{\operatorname{N^1}}}

\newcommand{\nbh}[1]{\operatorname{ngb}(#1)}

\def\Ad{\operatorname{Ad}}

\DeclareMathOperator{\No}{N}

\newcommand{\pb}{\ar@{}[dr]|{\text{\pigpenfont J}}}
\def\ol{\overline}

\newcommand{\shse}[3]{0 ~\ra ~#1~ \lra ~#2~ \lra ~#3~ \ra~ 0}
\makeatletter
\newcommand{\xleftrightarrow}[2][]{\ext@arrow 3359\leftrightarrowfill@{#1}{#2}}
\makeatother
\newcommand{\xdasharrow}[2][->]{
\tikz[baseline=-\the\dimexpr\fontdimen22\textfont2\relax]{
\node[anchor=south,font=\scriptsize, inner ysep=1.5pt,outer xsep=2.2pt](x){#2};
\draw[shorten <=3.4pt,shorten >=3.4pt,dashed,#1](x.south west)--(x.south east);
}}

\newcommand\ra{{\ \rightarrow\ }}
\newcommand\lra{\longrightarrow}

\def\Mo{\operatorname{\hspace{0cm}M}}



%% file: intro.tex

\section{Introduction}\label{sec:intro}

A remarkable idea present in modern Birational Geometry is the one that two birationally equivalent varieties should be obtained as GIT quotients of the same algebraic variety. This link between Geometric Invariant Theory and the Minimal Model Program arose upon the seminal work of Thaddeus and Reid in the 1990's (see \cite{Thaddeus1996,ReidFlip}) and leds to the concept of Mori Dream Spaces, whose small $\Q$-factorial modifications can be obtained as quotients by the action of a torus of the spectrum of a finitely generate algebra, called Cox ring of the Mori Dream Space. 

On the other hand, the action of a complex $1$-dimensional torus $\C^*$ on a projective variety induces birational transformations between the associated geometric quotients (cf. \cite[Lemma~3.4]{WORS1}, \cite[Remark~2.13]{WORS3}). It  makes then sense to ask whether any birational transformation may be obtained in this way. On one hand, this question must be tackled theoretically by considering $\C^*$-equivariant projective compactifications of the cobordisms studied by Morelli and W{\l}odarczyk (cf. \cite{Morelli,Wlodarczyk}).
On the other, one would like to construct explicitly a projective variety with an action of $\C^*$ inducing a given birational map $\psi$ --that we call a {\em geometric realization of $\psi$}--, which is the main motivation of this paper. This concept has been introduced in \cite{WORS4a}, where geometric realizations for a certain simple class of birational transformations, called bispecial, have been constructed; a similar construction had been proved to work in the case of Atiyah flips (cf. \cite{WORS2}). 

Another possible idea to construct geometric realizations is to start with some reasonable class of $\C^*$-actions on projective varieties, and then to study their $\C^*$-invariant subvarieties and the associated birational maps.   
In this paper we consider $\C^*$-actions on rational homogeneous spaces, and describe the birational maps induced by them. We stick to the case in which the action is equalized (see \ref{sssec:equalized} below), a condition that guarantees the smoothness of the birationally equivalent varieties associated to the action.  Remarkably the list of transformations that we obtain in this way contains many classically interesting examples, such as the inversion of matrices, and the special quadro-quadric Cremona transformations (see \cite{ESB}). Furthermore, by considering certain $\C^*$-invariant subsets of rational homogeneous spaces we are able to produce geometric realizations of other birational transformations, such as the special transformations of type $(2,1)$ classified by Fu and Hwang, \cite{FH}.\par\medskip

\noindent{\bf Outline.} After recalling some basic facts on torus actions on rational homogeneous varieties (Section \ref{sec:prelim}), we focus on the classification of equalized actions with isolated extremal fixed point components, and study their associated birational transformations (Section \ref{sec:examples}). Section \ref{sec:examples2}
contains projective descriptions of the Cremona transformations obtained in this way. In particular our discussion provides (among other presentation of rational homogeneous varieties as geometric realizations of birational maps) the following statement:
\begin{theorem*}\label{thm*:inversion}
Let $\psi:\P(M^\vee_{n\times n}(\C))\dashrightarrow \P(M^\vee_{n\times n}(\C))$ be the projectivization of the inversion map of $n\times n$ matrices, and let $\psi_s$, $\psi_a$ be its restrictions to the projetivizations of the spaces of symmetric and skew-symmetric matrices, respectively. Then there exist geometric realizations of $\psi$, $\psi_s$, $\psi_a$ (with $n$ even in this case), given by an equalized $\C^*$-action on, respectively, the following rational homogeneous varieties (see Section \ref{ssec:notation} for the notation):
$$
\DA_{2n-1}(n),\qquad \DC_n(n), \qquad \DD_n(n).
$$
\end{theorem*}

Note that the inversion maps 
$\psi$, $\psi_s$, $\psi_a$ induced by the equalized $\C^*$-action on $\DA_{5}(3), \DC_3(3), \DD_6(6)$ are quadro-quadric special Cremona transformations as described in \cite{ESB}; the list is completed with the map induced by the equalized $\C^*$-action on $\DE_7(7)$, see Remark \ref{rem:bispecial}.

In the final Section \ref{sec:derived} we study induced $\mathbb{C}^*$-actions on certain fiber bundles contained in rational homogeneous varieties, and show that some of them are geometric realizations of special birational transformations. More concretely, our arguments in that section allow us to pose the following  statement:
\begin{theorem*}\label{thm*:bundles}
Let $F$ be a Fano manifold of Picard number one, and assume that there exists a special birational transformation $\psi:F\dashrightarrow \P^{\dim(F)}$ of type $(1,2)$. Then there exists a geometric  realization of $\psi$, given by an equalized $\C^*$-action on a locally trivial $F$-bundle $X\to\P^1$. Furthermore, $F$ can be $\C^*$-equivariantly embedded in a rational homogeneous variety endowed with an equalized $\C^*$-action. 
\end{theorem*}

%% file: prelim.tex

\section{Preliminaries on $\C^*$-actions}\label{sec:prelim}

Throughout the paper all the varieties will be defined over the field of complex numbers. This section contains some preliminary results, notation and conventions regarding $\C^*$-actions. We refer the interested reader to \cite{BB,CARRELL,RW,WORS1,WORS3,B_R} for details, original references, and further results. In particular we will recall the concept of geometric realization of a birational map, introduced in \cite{WORS4a}. We will stick to the case in which the variety $X$ admitting the $\C^*$-action is smooth and projective, as this will be the situation in the remaining sections.

\subsection{Generalities on $\C^*$-actions}\label{ssec:torusactions}
Let $X$ be a complex, smooth, projective variety admitting a non-trivial (morphical) action of the multiplicative group $\C^*$, that we will denote by:
$$
\C^*\times X\lra X \qquad (t,x)\mapsto tx.
$$

\subsubsection{} $X^{\C^*}\subset X$ will denote the set of fixed points of the $\C^*$-action, which is a closed, smooth subset of $X$, and $\cY$ the set of irreducible components of $X^{\C^*}$.

\subsubsection{} For every $x\in X$ we set $$x^{\pm}:=\lim_{t^{\pm 1}\to 0} tx \in X^{\C^*}$$ and call them the {\em sink} and the {\em source of the orbit} $\C^*x$, respectively. The rest of the fixed point components are called {\em inner}, and their set is denoted by $\cY^\circ$.

\subsubsection{} $Y_-,Y_+\in \cY$ will denote the {\em sink} and the {\em source} of the action, defined as the fixed point components satisfying that:
$$
x^{\pm}\in Y_{\pm}
$$ 
for the general point $x\in X$. Sometimes we will call the the extremal fixed point components of the action.

\subsubsection{} For every $Y\in \cY$, we denote by $X^\pm(Y)\subset X$ the {\em positive/negative BB-cell} (short form of Bia{\l}ynicki-Birula cell) associated to $Y$, defined as the $\C^*$-invariant subsets:
$$
X^\pm(Y):=\{x\in X|\,\,x^{\pm}\in Y\}.
$$

\subsubsection{}\label{sssec:BB} For every $Y\in \cY$, we denote by $\cN_{Y|X}$ the normal bundle of $Y$ in $X$. The action of $\C^*$ on $X$ induces a fiberwise linear $\C^*$-action on the fibers of $\cN_{Y|X}$ over $Y$, and we may decompose it as a direct sum of vector bundles 
$$\cN_{Y|X}=\cN^-(Y,X)\oplus \cN^+(Y,X),$$ 
according to the sign of the weights of the $\C^*$-action.
When the ambient variety $X$ is clear in the context, we will simply denote $\cN^\pm(Y):=\cN^\pm(Y,X)$. By the Bia{\l}ynicki-Birula theorem, we know that we have a $\C^*$-equivariant isomorphism $$\cN^\pm(Y)\simeq X^\pm(Y),$$ for every $Y\in \cY$. The rank of $\cN^\pm(Y)$, which is equal to $\dim(X^\pm(Y))-\dim(Y)$, will be denoted by $\nu^{\pm}(Y)$. Note that 
\begin{equation}\label{eq:nus}
\begin{array}{l}\nu^\mp(Y_\pm)=0,\\[2pt]\nu^+(Y)+\nu^-(Y)=\dim(X)-\dim(Y),\mbox{ for every }Y\in \cY,\\[2pt]
\nu^\pm(Y)=0 \mbox{ if and only if }Y=Y_\mp.
\end{array}\end{equation}
for every $Y\in \cY$.

\subsubsection{} A {\em linearization} of the action of $\C^*$ on a line bundle $L$ on $X$ is a fiberwise linear $\C^*$-action of $\C^*$ on $L$ such that the natural projection is $\C^*$-equivariant.  Given $Y\in \cY$, $\C^*$ acts on $L_{|Y}$ by multiplication with $\mu_L(Y)\in \Mo(\C^*):=\Hom(\C^*,\C^*)=\Z$,  called the {\em weight of the linearization on $Y$}. 

\subsubsection{} If $L$ is ample, the $\C^*$-action together with a linearization on $L$ will be referred to as a {\em $\C^*$-action on the polarized pair $(X,L)$}. In this case the minimum and maximum value of the weights are achieved at the sink and the source of the action, respectively. By multiplying the linearization with a character we may assume that $\mu_L(Y_-)=0$, and then $\delta:=\mu_L(Y_+)$ is called the {\em bandwidth} of the $\C^*$-action on $(X,L)$. Furthermore, the set 
$$\mu_L(\cY)=\{a_0=0,a_1,\dots,a_r=\delta\},\quad (a_0<a_1<\dots<a_r)$$
is called the set of {\em critical values} of the $\C^*$-action on $(X,L)$. The integer $r$ is called {\em criticality} of the action.

\subsubsection{} A linearization of the $\CC^*$-action on an ample line bundle $L$ over $X$ gives a weight decomposition on $\HH^0(X,L)$, that we write as: 
$$\HH^0(X,L)=\bigoplus_{u\in\Z}\HH^0(X,L)_u,$$
where $\HH^0(X,L)_u\subset \HH^0(X,L)$ stands for the vector subspace of $\C^*$-weight $u$.
If $L$ is globally generated, the extremal values for which $\HH^0(X,L)_u\neq \{0\}$ are $\mu_L(Y_-)=0$, $\mu_L(Y_+)=\delta$. 

\subsubsection{}\label{sssec:equalized}  The action of $\C^*$ on $X$ is said to be {\em equalized} if and only if the induced action on $\cN^\pm(Y)$ has all its weights equal to $\pm 1$; that is to say that the action on $\cN^\pm(Y)$ is faithful and fiberwise homothetical, for every $Y\in \cY$. Equivalently, the $\C^*$-action has no proper non-trivial isotropy groups.

\subsubsection{}\label{sssec:AMvsFM} If the $\C^*$-action is equalized, then the closure of any $1$-dim\-ens\-ional orbit is a smooth rational curve, whose $L$-degree (with respect to $L\in\Pic(X)$ ample) can be computed in terms of the weights at its extremal points: 
\begin{equation}\tag{AM vs FM}
L\cdot \overline{\C^*x}=\mu_L(x_+)-\mu_L(x_-).
\end{equation}

\subsubsection{} The action of $\C^*$ on $X$ is said to be of {\em B-type} if and only if $\nu^\pm(Y_\pm)=1$, i.e. if $Y_\pm$ are divisors in $X$.

\subsubsection{}\label{sssec:blowup} If the $\C^*$-action on $X$ is equalized, then it extends to a B-type $\C^*$-action on the blowup 
$$
\beta:X^\flat\to X,
$$
of $X$ along $Y_\pm$, whose sink and source are the exceptional divisors $\P_{Y_\pm}(\cN_{Y_\pm|X}^\vee)$.

\subsubsection{}\label{sssec:induced} If the $\C^*$-action on $X$ is equalized, then the Bia{\l}ynicki-Birula theorem implies that $\P_{Y_\pm}(\cN_{Y_\pm|X}^\vee)$ is a geometric quotient of the open set $X^{\pm}(Y_\pm)\subset X$. In particular, the intersection of these two open sets induces a birational map:
\begin{equation}\label{eq:birinduced}
\psi:\P_{Y_-}(\cN_{Y_-|X}^\vee)\dashrightarrow \P_{Y_+}(\cN_{Y_+|X}^\vee).
\end{equation}
Note that in the case in which the $\C^*$-action is of B-type, the map $\psi$ goes from $Y_-$ to $Y_+$. Note also that the birational maps induced by an equalized action on $X$ and on its blowup along its sink and source are the same.

\subsubsection{} Conversely, if $\psi:E_-\dashrightarrow E_+$ is a birational map, a variety $X$ together with a $\C^*$-action such that we have two isomorphisms $E_\pm\simeq \P_{Y_\pm}(\cN_{Y_\pm|X}^\vee)$ so that the induced birational map from $E_-$ to $E_+$ coincides with $\psi$ is called a {\em geometric realization} of $\psi$.

\subsection{Exceptional locus of the induced birational transformation}\label{ssec:Btype}

In this section we will describe the exceptional locus of the birational map associated to a faithful $\C^*$-action on a smooth projective variety $X$. For simplicity we will assume that the action is of B-type (see \ref{sssec:blowup} above), so that the $\C^*$-action induces a birational map:
$$
\psi:Y_-\dashrightarrow Y_+,$$
which assigns to a general point $y\in Y_{-}$ the limit for $t \to 0$ of the unique orbit having limit for $t^{-1}\to 0$ equal to $y$.

Given an inner fixed point component $Y\in \cY^\circ$ we consider the invariant closed subvarieties of $X$ defined recursively as follows:
$$
C_1^\pm(Y):=\ol{X^\pm(Y)},\qquad C_k^\pm(Y):=\ol{X^\pm(C_{k-1}^\pm(Y))}, \,\,k\geq 2
$$
Note that $C_{k-1}^\pm(Y)\subseteq C_k^\pm(Y)$, and that the inclusion is strict if and only if there exists a $1$-dimensional orbit $\C^*x$ not contained in $C_{k-1}^{\pm}(Y)$ such that $\lim_{t\to 0}t^{\pm 1}x\in C_{k-1}^\pm(Y)$. Then  for every inner fixed point component $Y$ there exists $m\geq 1$ such that $C_{m-1}^\pm(Y)= C_m^\pm(Y)$, and so we may define:
$$
C^{\pm}(Y):=\lim_{k\to+\infty}C_k^\pm(Y),\,\, Y\in\cY^\circ.
$$
In other words,  
$C^{\pm}(Y)$ can be defined as the union of all the connected $1$-cycles consisting of closures of $1$-dimensional orbits linking $Y$ with $Y_{\mp}$, whose existence is guaranteed by \cite[Corollary 2.15]{WORS1}. We finally set:
$$
Z_\mp(Y):=C^{\pm}(Y)\cap Y_{\mp}\subset Y_{\mp}, \,\,Y\in\cY^\circ,\quad Z_\mp:=\bigcup_{Y\in \cY^\circ}Z_\mp(Y)\subset Y_{\mp}.
$$

\begin{lemma} \label{lem:birational_map_b} In the above notation, the map $\psi$ restricts to an isomorphism:
$$
\psi:Y_{-}\setminus Z_{-}\lra Y_{+}\setminus Z_+,
$$
\end{lemma}

\begin{proof}
Denote by $\pi_\pm:X^\pm(Y_\pm)\simeq \cN_{Y_\pm|X}\to Y_\pm$ the projection map given by the Bia{\l}ynicki-Birula theorem. The open subsets:
$$U_\pm=X^\pm(Y_\pm)\setminus \big(Y_\pm \cup \pi_\pm^{-1}(Z_\pm)\big)$$ are equal since, by definition, they can be identified with the set of orbits of the action having limiting points at $Y_{-}$ and $Y_{+}$. It then follows that their quotients by the action of $\C^*$, $Y_-\setminus Z_-$ and $Y_+\setminus Z_+$,  are isomorphic.
\end{proof}

This statement tells us that the exceptional locus of $\psi$ is contained in $$\bigcup_{Y\in\cY^\circ} Z_{-}(Y).$$ However, in general $\Exc(\psi)$ does not coincide with this set; in fact one may show the following: 

\begin{lemma}\label{lem:defincodim1}
Let $X$ be a smooth, projective variety together with an equalized B-type $\C^*$-action, and let $\cC$ be the set of inner fixed point components $Y$ of the action with  $\nu^{-}(Y)>1$. Then:
$$
\Exc(\psi)\subseteq\bigcup_{Y \in \cC} Z_-(Y).
$$
\end{lemma}

\begin{proof} 
Let $x_-\in Y_-\setminus \bigcup_{Y \in \cC} Z_-(Y)$. We start by claiming that there exist a unique sequence $\Gamma_1,\Gamma_2,\dots,\Gamma_r$ of closures of $1$-dimensional orbits 
such that $x_0:=x_-\in \Gamma_1$, every $\Gamma_i$ intersects $\Gamma_{i+1}$ at a fixed point $x_i$, $\Gamma_r$ intersects $Y_+$ at a point $x_r=x_+$, and $\mu_L(x_i)<\mu_L(x_{i+1})$ for every $i$. 

In order to prove the claim, we note first that, by the Bia\l ynicki-Birula decomposition, there exists a unique $1$-dimensional $\C^*$-orbit  having $x_-$ as sink; let us denote by $\Gamma_1$ its closure, by $x_1$ its source, and by $Y_1$ the unique fixed point component containing $x_1$. If there were at least two $1$-dimensional orbits having $x_1$ as sink, the Bia\l ynicki-Birula decomposition would tell us that $\nu^-(Y_1)>1$, that is $Y_1\in\cC$; but then $x_-\in Z_-(Y_1)$ would belong to  $\bigcup_{Y \in \cC} Z_-(Y)$, a contradiction. We proceed now recursively, proving that there exists a unique $1$-dimensional orbit having $x_1$ as sink, whose closure is denoted $\Gamma_2$, etc. The fact that the weight $\mu_L(x_i)$ grows at every step follows by \ref{sssec:AMvsFM}.

We now prove that the map $\psi$ extends to $x_-$. To this end, let us consider the unique sequence of $\C^*$-invariant curves $\Gamma_1,\dots,\Gamma_r$ linking $x_-$ with $x_+\in Y_+$ as above. Let $\gamma(s)$ be a holomorphic curve converging to $x_-$ when $s$ goes to $0$, with $\gamma(s)\in Y_-\setminus Z_-$ for $s\neq 0$. For every $s\neq 0$, the image $\psi(\gamma(s))$ is defined as the source of the unique orbit $O(s)$ having $\gamma(s)$ as sink when $t$ goes to $\infty$. When $s$ goes to $0$, the closure  $\overline{O(s)}$ must converge to an effective and connected  $\C^*$-invariant $1$-cycle $\Gamma'$ passing by $x_-$ and meeting $Y_+$; moreover, by the AM vs FM formula (\ref{sssec:AMvsFM}) we know that $L \cdot \Gamma' = L \cdot \overline{O(s)}$ equals the bandwidth $\delta$. In particular it contains a $1$-cycle $\sum_{i=1}^m r_i \Gamma_i'$ with $x_- \in \Gamma_1'$, $\Gamma_i' \cap \Gamma_{i+1}' =\{x_i'\}, \Gamma_m' \cap Y_+ \not = \emptyset$.
By \ref{sssec:AMvsFM}, the equalization of the action, and the equality $L \cdot \Gamma' = \delta$,  we may conclude that $r_i=1$ for every $i$, that $\Gamma' = \sum_{i=1}^m \Gamma_i'$ and that $\mu_L(x_i') < \mu_L(x_{i+1}')$ for every $i$. In other words, $\Gamma'$ equals $\Gamma_1+\dots+\Gamma_r$, and from this it follows that 
$$\lim_{s\to 0}\psi(\gamma(s))=(\Gamma_1+\dots+\Gamma_r)\cap Y_+=x_+.$$
In particular, since this limit does not depend on the choice of the curve $\gamma$, it follows that the map $\psi$ extends to $x_-$. This finishes the proof.
\end{proof}

\subsection{Local description of $\C^*$-invariant divisors}\label{ssec:local}

We end this section by describing locally the invariant divisors determined by sections $\sigma\in \HH^0(X,L)_u$, paying attention to their smoothness. The following result is a restatement of \cite[Lemma~2.17]{BWW}; the notation we will use is consistent with that paper.

\begin{lemma}\label{lem:localdes}
 Let $X$ be a smooth proper variety together with a $\C^*$-action, linearized on a line bundle $L$ on $X$. Let $Y\subset X$ be an irreducible fixed point component, and $y\in Y$ be a point. Let $D_\sigma\subset X$ be a $\CC^*$-invariant divisor associated to a $\CC^*$-equivariant section $\sigma\in\HH^0(X,L)_u$, with $u\in \ZZ$. Then there exists local coordinates $x_1,\dots, x_{d_+},y_1,\dots, y_{d_0}, z_1,\dots,z_{d_-}$ around $y=\{x_i=y_j=z_k=0\}$ such that $Y=\{x_i=z_k=0\}$, and the divisor $D_\sigma$ is described as the zero set of a power series $f\in\CC[[x_i,y_j,z_k]]$, homogeneous  of degree $\mu_L(Y)-u$ with respect to the grading of $\CC[[x_i,y_j,z_k]]$ induced by the $\C^*$-action.
\end{lemma}

\begin{proof}
Without loss of generality, we may assume that $\mu_L(Y)>0$. We consider the $\C^*$-action on the  line bundle $L$, identifying $X$ and $Y$ with their corresponding images into $L$ via the zero section. Note that in this way $X\subset L$ is a $\C^*$-invariant subvariety and, by the hypothesis on $\mu_L$, $Y\subset L$ is a $\C^*$-fixed point component. 

Given a point $y\in Y$, we may apply \cite[Theorem~2.5]{BB} to find a $\C^*$-invariant neighborhood $U$ of $y\in L$ equivariantly isomorphic to $(Y\cap U)\times V$, where $V$ is a $\C^*$-module. By differentiating at $y$, $V$ may be identified with $\cN_{Y|L,y}=\cN_{Y|X,y}\oplus L_{y}$. The action of $\C^*$ on $V$ diagonalizes with weights $p_1,\dots, p_{d_+}, n_1,\dots,n_{d_-},\mu_L(Y)$ ($p_i>0, n_k<0$), on a set of coordinates $(x_1,\dots, x_{d_+}, z_1,\dots,z_{d_-},\lambda)$. 
We may complete this set of coordinates with local coordinates $y_1,\dots, y_{d_0}$ of $Y$ around $y$, and consider the rings of local coordinates $\C[[x_i,y_j,z_k]]$ of $X$ around $y$, and $\C[[x_i,y_j,z_k,\lambda]]$ of $L$ around $y$.

In these local  coordinates a section $\sigma \in \HH^0(X,L)$ reads as $\sigma(x_i,y_j,z_k)=(x_i,y_j,z_k, f(x_i,y_j,z_k))$ where $f\in\CC[[x_i,y_j,z_k]]$.
The proof is then concluded by noting that saying that $\sigma$ is invariant of weight $u$ means that 
$$f(t(x_i,y_j,z_k))=t^{\mu_L(Y)-u}f(x_i,y_j,z_k).$$
\end{proof}

A straightforward consequence of the above Lemma is the following:

\begin{corollary}\label{cor:localdes}
In the situation of the above lemma suppose that the action is equalized. Then the function $f$ vanishes along $Y$ with multiplicity $\geq |u-\mu_L(Y)|$. If moreover $Y$ is either the sink or source, then the inequality becomes an equality.
\end{corollary}

\begin{proof}
By  Lemma \ref{lem:localdes}, the function $f$ can be written as a series:
$$
f=\sum_{I,J,K} a_{IJK}x^Iy^Jz^K
$$
where 
$$
x^I=x_1^{i_1}\dots x_{d_+}^{i_{d_+}},\quad y^J=y_1^{j_1}\dots y_{d_+}^{j_{d_0}},\quad z^K=z_1^{k_1}\dots z_{d_-}^{k_{d_-}},
$$
homogeneous of degree $u-\mu_L(Y)$ with respect to the action of $\C^*$. Set $\ol{I}$, $\ol{K}$ to be the sum of the exponents in $I,K$, respectively. Since the action is equalized, we have that $\deg x_i=1$, $\deg z_k=-1$ for all $i$ and $k$, and so 
$$
u-\mu_L(Y)=\deg(f)=\deg(x^Iy^Jz^K)=\ol{I}-\ol{K}, \mbox{ for }a_{IJK}\neq 0.
$$ 
Then the vanishing multiplicity of $f$ along $Y$, which is $\min\{\ol{I}+\ol{K},\,\,a_{IJK}\neq 0\}$ is bigger than or equal to $|\ol{I}-\ol{K}|=|u-\mu_L(Y)|$. For the second part we simply note that if $Y$ is an extremal component, then either $I$ or $K$ are empty. 
\end{proof}

The following statement is an application of the above result to a particular situation, in the form of a smoothness criterion for invariant sections of $L$.

\begin{proposition}\label{prop:smooth}
Let $(X,L)$ be a smooth polarized pair with an equalized $\C^*$-action of bandwidth two. Let $\sigma\in\HH^0(X,L)_1$ be a section, with zero locus $D_\sigma$, such that $D_\sigma$ is smooth at the points of $D_\sigma\cap Y_1$. Then the differential $d\sigma$ defines sections of $N^\vee_{Y_\pm/X}\otimes L$; if these sections are nowhere vanishing then $D_\sigma\subset X$ is smooth and the induced action on $(D_\sigma,L_{|D_\sigma})$  has bandwidth two. 
\end{proposition}

\begin{remark}\label{rem:smooth}
The condition on the non-vanishing of the section of $N^\vee_{Y_\pm/X}\otimes L$ defined by $\sigma$ is satisfied for a general $\sigma$ if $L$ is very ample and, either $\dim Y_\pm<\frac{1}{2}\dim X$, or $\dim Y_\pm=\frac{1}{2}\dim X$ and the top Chern class of $N^\vee_{Y_\pm}\otimes L$ is zero. 
\end{remark}

\begin{proof}[Proof of Proposition \ref{prop:smooth}]
Since the section $\sigma$ is $\C^*$-invariant, it is enough to
check its smoothness at the fixed point locus. In fact, the singular locus of the section is closed and $\C^*$-invariant; then if it is non-empty, it contains a singular fixed point.  We will check smoothness locally so that we can write the expansion of $\sigma$ in coordinates.

By assumption we just need to check the smoothness of $D_\sigma$  at its intersection with the source (smoothness at the sink is analogous). We then take $y$ to be a point of the source $Y_+$, and use Lemma \ref{lem:localdes} to obtain local coordinates 
$y_i's$  of weight zero and $z_j's$ of weight one. Therefore, if
 $f$ is a local description of $\sigma$, it vanishes with multiplicity one at $Y_+$, by Corollary \ref{cor:localdes}. In particular, $df_{|Y_+}$ is not identically zero, and the zero locus of $f$ is singular at the points of $Y_+$ on which this differential is zero. 

 Moreover, since $f$ vanishes at $Y_+$, we may write:
 $$
 df_{|Y_+}=\sum_k f_k(y_1,\dots,y_{d_0}) dz_k.
$$
In other words $df_{|Y_+}$, as a section of $(\Omega_X\otimes L)_{|Y_+}$, lies in the kernel of the differential map $(\Omega_X\otimes L)_{|Y_+}\to \Omega_{Y_+}\otimes L_{|Y_+}$, which is $N^\vee_{{Y_+}/X}\otimes L_{|Y_+}$. This finishes the proof.
\end{proof}

%% file: examples.tex

\section{Torus actions on rational homogeneous spaces}\label{sec:examples}

Representation theory provides a framework in which one may construct many examples of projective varieties supporting $\C^*$-actions (see \cite[II, Chapter 3]{BBC}), that can be then thought of as geometric realizations of some birational transformations.  This Section is devoted to the description of  
$\C^*$-actions on rational homogeneous varieties. We will pay special attention to the ones that have isolated extremal fixed points, that give rise to Cremona transformations whenever they are equalized (that we will study in Section \ref{sec:examples2} below). In particular, we will show how to re-construct in this setting the geometric realizations of the special Cremona transformations within the list of Ein and Shepherd-Barron. 

\subsection{Notation and basic facts on rational homogeneous varieties}\label{ssec:notation}

In this section we will recall some basic facts on rational homogeneous varieties, and introduce the notation we will use to describe them. We will use some standard representation theory of semisimple algebraic groups, for which we refer the interested reader to \cite{Hum,Hum1}. 

\subsubsection*{Rational homogeneous varieties}\label{sssec:RHvar}

Let $G$ be a semisimple algebraic group and $H\subset G$ a Cartan subgroup, with associated Lie algebras $\fh\subset\fg$; assume that $G$ is the adjoint group of $\fg$, so that the group $\Mo(H):=\Hom(H,\Z)$ of characters of $H$ coincides with the root lattice of $G$ with respect to $H$, generated by the root system $\Phi\subset\Mo(H)$ of $\fg$  (with respect to $H$). We will denote by $\Delta=\{\alpha_1,\dots,\alpha_n\}\subset \Phi$ a base of positive simple roots of $\Phi$, and by $\Phi^+\subset \Phi$ the set of roots that are non-negative integral combinations of elements of $\Delta$, so that $\fb=\fh\oplus\bigoplus_{\alpha\in \Phi^+}\fg_\alpha$ is a Borel subalgebra of $\fg$, corresponding to a Borel subgroup $B\subset G$ containing $H$. It is then known that any projective variety admitting a transitive action of $G$ (a so-called {\em rational homogeneous $G$-variety}) can be written as $G/P$, where $P$ is a subgroup of $G$ containing $B$. 

The Weyl group of $G$ is defined as the finite group $W:=\No_G(H)/H$; its natural action on $\Mo(H)\otimes_{\Z}\R$  preserves the inner product induced by the Killing form $\kappa(\cdot,\cdot)$ of $\fg$ (which gives $\Mo(H)\otimes_{\Z}\R$ the structure of Euclidean space). In this way $W$ can be described as the group generated by the reflections $r_i$ with respect to the positive simple roots $\alpha_i$. Given $w\in W$, the minimum number $k$ of reflections $r_i$ such that we can write $w=r_{i_1}\circ\dots\circ r_{i_k}$ is called the {\em length} of $w$; it is known that there exists a unique element of $W$ of maximal length, that is called the {\em longest element} of $W$, and that we will denote by $w_0$. We will denote by $\cD$ the Dynkin diagram of $G$, and by $D$ its set of nodes, which is in one to one correspondence with the base of positive simple roots $\Delta$. 

\begin{remark}\label{rem:longest}
The bijection of the root system $\Phi$ given by $w_0$, in the cases in which $\fg$ is simple, is described in \cite[Planche I--IX]{Bourb}. Essentially $w_0$ equals $-\id$ whenever $\cD$ has no non-trivial automorphisms ($\DB_n,\DC_n,\DE_7,\DE_8,\DF_4,\DG_2$) and in the case $\DD_n$, $n$ even; in the cases $\DA_n, \DE_6$ and $\DD_n$, $n$ odd, it is the composition of $-\id$ with the homomorphism of $\Mo(H)$ determined by the permutation of $\Delta$ induced by the non-trivial automorphism of $\cD$.
\end{remark}

Every rational homogeneous $G$-variety $G/P$ is determined by the marking of $\cD$ on a set of nodes $\{i_1,\dots,i_k\}\subset D$ (corresponding to positive simple roots $\alpha_{i_1},\dots,\alpha_{i_k}$), where $k$ equals the Picard number of $X$; we refer to \cite[Section~2]{MOSWW} for details. We will then set:
$$
\cD(I):=G/P, \mbox{ for }I=\{i_1,\dots,i_k\}.
$$
The parabolic subgroup $P$ associated to $I$ can be described as $P=BW(D\setminus I)B\supset B$, where $W(D\setminus I)\subset W$ is the subgroup of $W$ generated by the reflections $r_j$, $j\in D\setminus I$.

With this notation, any inclusion $I\subset I'\subset D$ gives rise to a natural projection 
$$
\cD(I')\to \cD(I), 
$$
whose fibers are rational homogeneous varieties of the form $\cD_I(I'\setminus I)$, where $\cD_I$ denotes the Dynkin subdiagram of $\cD$ obtained by deleting the nodes corresponding to the indices in $I$. 

The marking of the Dynkin diagram $\cD$, of a semisimple group $G$ as above, on the whole set of nodes $D$ corresponds to the quotient of $G$ by the Borel subgroup $B$ containing $H$, which is usually called the complete flag variety associated to $G$.
Complete flags can be characterized (cf. \cite{OSWi}) by the property of being smooth projective varieties having as many independent $\P^1$-bundle structures as their Picard number; in the notation we have just introduced, these structures are precisely the natural projections:
$$
\cD(D)\to \cD(D\setminus \{i\}), \qquad i=1,\dots n.
$$
Later on we will denote the numerical class of the fibers of these maps by $\Gamma_i$.

\subsubsection*{Examples of rational homogeneous varieties}\label{sssec:exRHvar}

For the connected Dynkin diagrams we will use the numbering proposed by Bourbaki (cite \cite[Planche I--IX]{Bourb}); for instance, the diagram $\DE_6$ is numbered as in Figure \ref{fig:E6}.
\begin{figure}[h!]
\begin{center}
\input{DiagE6num.tex}
\caption{The Dynkin diagram $\DE_6$.\label{fig:E6}}
\end{center}
\end{figure}
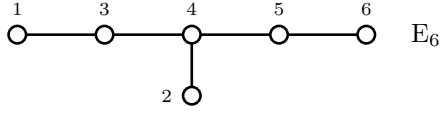
 
Algebraic geometers usually denote rational homogeneous varieties in a way that it reflects one of  their geometric descriptions.  The advantage of the above ``group-minded'' notation comes from the fact that it may be applied to all rational homogeneous varieties. 
For the reader's convenience, we include here a list of rational homogeneous varieties, their projective descriptions, and the notation we will use for them. 
\begin{itemize}[leftmargin=*]
\item The projective space $\P^n$ is written as $\DA_n(1)$, and its dual as $\DA_n(n)$; more generally, the Grassmannian of $(k-1)$-linear subspaces of $\P^n$ is $\DA_n(k)$, and so $\DA_n(k_1,\dots,k_r)$, $k_1<\dots< k_r$ is the variety of flags of linear spaces $\P^{k_1-1}\subset\dots \subset \P^{k_r-1}\subset \P^n$.
\item Smooth quadrics of dimension $2n-1$ (resp. $2n-2$) are denoted by $\DB_n(1)$ (resp. $\DD_n(1)$). Varieties of the form $\DB_n(k_1,\dots, k_r)$ parametrize flags of linear subspaces of $\DB_n(1)$. The case of flags in $\DD_n(1)$ is essentially analogous, but there exist two disjoint irreducible families of $(n-1)$-dimensional linear subspaces, denoted by $\DD_n(n-1)$, $\DD_n(n)$ (the so-called spinor varieties); the family of $(n-2)$-dimensional linear subspaces in $\DD_n(1)$ is $\DD_n(n-1,n)$.
\item Varieties of the form $\DC_n(k_1,\dots, k_r)$ parametrize flags of linear subspaces in $\P(V)=\P^{2n-1}$ that are isotropic with respect to a maximal rank skew-symmetric form in $V$. The variety $\DC_n(1)$ is nothing but $\P(V)$, and the varieties $\DC_n(k)$ are usually called isotropic Grassmannians. 
\item The rational homogeneous varieties of types $\DE_k$, $\DF_4$, $\DG_2$ can be described projectively in terms of the algebra of complexified octonions; we refer the interested reader to \cite{LM} and the references therein. Here it will be enough to remark that the variety $\DE_6(1)$ is the  octonionic projective plane, sometimes called the Cartan variety. It can also be described as the Severi variety of dimension $16$ in Zak's classification, \cite[IV. Theorem 4.7]{Z}. The rest of varieties of type $\DE_6$ can be seen as families of subvarieties (or flags of subvarieties) in $\DE_6(1)$. For instance $\DE_6(6)$ (which is isomorphic to $\DE_6(1)$) parametrizes smooth $8$-dimensional quadrics in $\DE_6(1)$; furthermore $\DE_6(3)$ parametrizes the lines in (the minimal projective embedding of) $\DE_6(1)$, so that we may assert, by our previous arguments, that the family of lines in $\DE_6(1)$ passing by a point is isomorphic to the spinor variety $\DD_5(5)$. In a similar way, the family of lines passing by a point in the variety $\DE_7(7)$ is isomorphic to $\DE_6(1)$. 
\end{itemize}
 
\subsubsection*{Line bundles on rational homogeneous varieties}\label{sssec:lbRHvar}

Let us finish this section by  describing the Picard group of the complete flag $\cD(D)$, that contains $\Pic(\cD(I))$ via the corresponding pullback map, for every $I\subset D$.

Let us denote by $G'$ the (unique) simply connected semisimple group with Lie algebra $\fg$. One has a surjective finite homomorphism of algebraic groups $\phi:G'\to G$ such that $B':=\phi^{-1}(B)$ is a Borel subgroup of $G'$ (so that $G'/B'\simeq G/B=\cD(D)$), and  $H':=\phi^{-1}(H)\subset G'$ is a maximal torus; moreover the restriction of $\phi$ to $H'$ induces an inclusion of $\Mo(H)$ into $\Mo(H')$ as a sublattice of finite index. Furthermore after embedding these two lattices in the Euclidean space $\Mo(H)\otimes_{\Z}\R=\Mo(H')\otimes_{\Z}\R$, one may identified $\Mo(H')$ with the lattice of abstract weights of $\fg$, defined as the set of vectors $v\in \Mo(H)\otimes_{\Z}\R$ satisfying:
$$
2\dfrac{\kappa(v,\alpha)}{\kappa(\alpha,\alpha)}\in\Z,\quad\mbox{for all }\alpha\in \Phi.
$$
For every element $\lambda\in \Mo(H')$ one may consider its composition with the projection $\xi:B'\to H'$ (defined by quotienting $B'$ by its unipotent radical) to obtain a homomorphism from $B'$ to $\C^*$ defining an action of $B'$ on $\C$ and, subsequently, a line bundle on $\cD(D)$:
$$
L:=G'\times^{B'}\C= (G'\times\C)/\sim,\quad (g',v)\sim (g'b',\xi(b')^{\lambda}v)\mbox{ for every $b'\in B'$.}
$$
Furthermore, one may prove (see for instance \cite[Theorem~6.4]{Snow}) that:
\begin{proposition}\label{lem:picflag}
With the notation as above, the correspondence $\lambda \mapsto L$ induces an isomorphism of lattices:
$$
 \Mo(H')\simeq\Pic(\cD(D)).
$$
\end{proposition}

Finally we will denote by $\{\lambda_1,\dots,\lambda_n\}\in\Mo(H')$ the set of fundamental weights of $\fg$, defined by:
$$
2\dfrac{\kappa(\alpha_i,\lambda_j)}{\kappa(\alpha_i,\alpha_i)}=\delta_{ij}, \mbox{ for every }i,j,
$$
and by $L_1,\dots,L_n\in\Pic(\cD(D))$ the line bundles associated to $\lambda_1,\dots,\lambda_n$ (which constitute a $\Z$-basis of $\Pic(\cD(D))$).  Note that we have:
$$
L_i\cdot \Gamma_j=\delta_{ij}, \quad \text{for every}~i,j.
$$
The numerical classes of the $L_i's$ (that we denote also by $L_i$) generate the nef cone $\Nef(\cD(D))\subset \NU(\cD(D))$ of $\cD(D)$, and more generally, denoting by $\pi_I:\cD(D)\to\cD(I)$ the natural projection, we may write:
$$
\pi_I^*(\Nef(\cD(I)))=\R_{\geq 0}\{L_i,\,\,i\in I\}.
$$ 

By abuse of notation, for every $i\in D$ we will denote also by $L_i$ the line bundle on $\cD(i)$ whose pullback to $\cD(D)$ is $L_i\in \Pic(\cD(D))$. They are known to be very ample and to generate $\Pic(\cD(i))$. The action of $\C^*$ on $\cD(i)$ extends to an action on the projective space $\P(\HH^0(\cD(i),L_i)^\vee)$, in which $\cD(i)$ is embedded. Although it is not true in general that the family of lines in $\cD(i)$ is $G$-homogeneous (this happens whenever $i$ is not an exposed short node of $\cD$, see \cite[Theorem 4.3]{LM}), there is a covering family of lines parametrized by a $G$-homogeneous projective variety; more concretely, one consider the set $\nbh{i}$ of nodes linked to $i$ in the Dynkin diagram $\cD$, so that the natural maps:
$$
\xymatrix@C=2mm{&\cD(\nbh{i}\cup \{i\})\ar[rd]\ar[ld]&\\\cD(\nbh{i})&&\cD(i)}
$$
can be regarded as a family of rational curves in $\cD(i)$ and the corresponding evaluation. This is called the family of {\em $G$-isotropic lines} in $\cD(i)$, and it can be described as the image in $\cD(i)$ of the fibers of the contraction $\cD(D)\to\cD(D\setminus\{i\})$; abusing notation, we will denote their numerical class by $\Gamma_i$. 

\subsection{$\C^*$-actions on rational homogeneous varieties}\label{ssec:torRH}

\subsubsection*{$\C^*$-actions vs. co-characters}\label{ssec:cochar}

Up to composition with a character, an action of $\C^*$ on a rational homogeneous variety $X$ can always be extended to the action of a maximal torus $H$ of an adjoint group $G$ acting transitively on $X$, so that we may write $X=G/P$ for some $P$, and assume (after conjugation) that $H\subset B\subset P$. In other words, $\C^*$-actions on $G/P$ are parametrized by group homomorphisms:
$$
\sigma:\Mo(H)\to \Z,
$$
called co-characters of $H$.
Note that by restricting this map to the root system $\Phi$ of $G$ one get a $\Z$-grading of the Lie algebra $\fg$:
$$
\fg=\fh\oplus\bigoplus_{\alpha\in\Phi}\fg_\alpha=\underbrace{\fh\oplus\bigoplus_{\sigma(\alpha)=0}\fg_\alpha}_{\fg_0}\oplus\bigoplus_{m\in\Z_{\neq 0}}\underbrace{\bigoplus_{\sigma(\alpha)=m}\fg_\alpha}_{\fg_m}
$$ 
The subspace $\fg_0\subset \fg$ is a Lie subalgebra of $\fg$, that is reductive.

In this paper, we will focus mostly on the case of $\C^*$-actions given by a particular type of co-characters, defined as follows. 
Given an index $i\in D$ we consider the grading in $\fg$ given by the unique map $\sigma_i$ sending the positive simple root $\alpha_i$ to $1$ and  $\alpha_j$, $j\neq 0$, to $0$. In this case, the subalgebra $\fg_0$ is the Lie subalgebra of a (reductive) subgroup $G_0$ that is a Levi part of a parabolic subgroup $P_i\subset G$, whose Lie algebra is 
$$
\fp_i=\fg_0\oplus \bigoplus_{m> 0}\fg_m.
$$
Note also that the subgroup $\C^*\subset H$ determined by $\sigma_i$ can be described as the center of the subgroup $G_0\subset G$.

\subsubsection*{The transversal group of a $\C^*$-action}\label{ssec:transversal}

Let $\sigma_i$ be a co-character of $H$ as above, and let $G_0\subset G$ be the corresponding subgroup determined by the associated $\Z$-grading of $\fg$.  
 
\begin{definition}\label{def:transversal}
With the above notation, the semisimple group $$G^\perp:=[G_0,G_0]\subset G$$ is called the {\em transversal subgroup} of the 
co-character $\sigma_i$.
\end{definition}

\begin{remark}\label{rem:commute}
Note that by construction, for every co-character $\sigma_i$ as above, $[\C^*,G^\perp]=\{1\}$. In particular, the action of $\C^*$ associated to $\sigma_i$ on any homogeneous manifold $G/P$ commutes with the action of $G^\perp$. In particular, the action of $G^\perp$ on $G/P$ sends $\C^*$-orbits to $\C^*$-orbits, and leaves the set of $\C^*$-fixed points $(G/P)^{\C^*}$ invariant; since moreover $G^\perp$ is connected, then every component of $\C^*$-fixed points is invariant by $G^\perp$. More precisely:
\end{remark}

\begin{proposition}\label{prop:fixcomp}
Let $X=G/P$ be a rational homogeneous variety, endowed with the $\C^*$-action associated to a co-character $\sigma_i$ of $G$ as above. Then the fixed point components of $X$ by the action of $\C^*$ are rational homogeneous $G^\perp$-varieties.
\end{proposition}

\begin{proof}
The proof follows the line of argumentation of \cite[Theorem~2.6]{Tev2}. We start by considering a very ample $G$-homogeneous line bundle $L$ on $G/P$, so that we have an embedding $G/P\subset \P(V)$ where $V=\HH^0(X,L)$, and the action of $G$ on $X$ is the restriction of the action of $G$ on $\P(V)$. In particular we get an action of $\C^*\subset H\subset G$ on $V$, that provides a $\Z$-grading:
$$
V=\bigoplus_{m\geq 0}V_m,\quad V_m:=\{v\in V|\,\,t(v)=t^mv\}.
$$
In particular $X^{\C^*}=\bigcup_{m\geq 0}\P(V_m)\cap X$.
Note that we may consider $V$ as a $\fg$-module, and then $V$ is a $\Z$-graded module over the $\Z$-graded Lie algebra $\fg$, so that we have 
\begin{equation}\label{eq:graded}
\fg_kV_m\subset V_{k+m},\quad\mbox{for every }m,k\in \Z.
\end{equation}

By Remark \ref{rem:commute}, every component of $\P(V_m)\cap X$ is $G^\perp$-invariant, for every $m$. 
The proof may then be completed by showing that: 
$$
T_{X\cap \P(V_m),x}=T_{G^{\perp}x,x}, \mbox{ for every }x\in \P(V_m)\cap X.
$$
Since $G_0=\C^*G^\perp=G^\perp\C^*$, and $x$ is fixed by $\C^*$, it is enough to show that $T_{X\cap \P(V_m),x}$ is equal to $T_{G_0x,x}$.
The point $x$ is the class modulo homotheties of a vector $v\in V_m^\vee$, so we have:
$$
T_{G_0x,x}=\dfrac{\fg_0 v}{\langle v\rangle}, \quad T_{X\cap \P(V_m),x}
=\dfrac{\fg v \cap V_m}{\langle v\rangle}.
$$
We conclude by noting that $\fg_0 v=\fg v \cap V_m$ by Equation (\ref{eq:graded}).
\end{proof}

\subsection{The action of a maximal torus on a complete flag manifold}\label{ssec:maxflag}

An important ingredient that we will use later on is the action of the Cartan subgroup $H\subset B\subset G$ on the complete flag $G/B$, whose properties are well known (cf. \cite[Section~3.4]{CARRELL}), since they are essentially related to the Bruhat decomposition of $G$. Then via the natural $G$-equivariant projections $G/B\to G/P$, one may study the action of $H$ (and of any subtorus in it) on any $G/P$. 

\subsubsection*{Fixed points of the Cartan action on $G/B$}\label{sssec:fixedpoints}

The following statement describes the fixed points of the $H$-action on $G/B$, and its infinitesimal behavior around those points; we include a proof for the reader's convenience.

\begin{proposition}\label{prop:Bruhat}
Let $G$ be a semisimple algebraic group. The set of fixed points of the action of a maximal torus $H\subset B\subset G$ on the complete flag $G/B$ is:
$$
(G/B)^H=\left\{wB|\,\,w\in W\right\}.
$$ 
For every $w\in W$, the set of weights of the $H$-action on the tangent space $T_{G/B,wB}$ is $\{w(\alpha)|\,\, -\alpha\in\Phi^+\}$. 
\end{proposition}

\begin{proof}
Note first that given an element $w=nH\in W$, the class $nB$ does not depend on the choice of $n$, so it makes sense to denote $nB$ by  $wB$, as we have done in the statement.
 
A point $gB\in G/B$ is $H$-fixed if and only if 
$g^{-1}Hg\subset B$. Since any two maximal tori in a connected solvable group are conjugated (cf. \cite[19.3]{Hum1}), there exists $b\in B$ such that $g^{-1}Hg=bHb^{-1}$, and we conclude that $gb\in \No_G(H)$. By setting $w=gbH\in W$, we may finally write $gB=gbB=wB$. 

For the second part it is enough to note that $T_{G/B,wB}$ is ($\C^*$-equivariantly) isomorphic to $\fg/\Ad_w(\fb)$, which decomposes in $\C^*$-eigenspaces as:
\[\fg/\Ad_w(\fb)=\sum_{-\alpha\in\Phi^+}\fg_{w(\alpha)}.
\]
\end{proof}

\subsubsection*{Weights of the Cartan action}\label{sssec:weightsCartan}

Now we will compute the weights of the $H$-action on $G/B$ on its set of fixed points, with respect to a given line bundle $L\in\Pic(G/B)$, which we identify with a character $\lambda\in \Mo(H')$ of the maximal torus of the universal covering $G'$ of $G$, see Proposition \ref{lem:picflag}. While the existence of a linearization of the action of $H$ on $L$ is garanted in a more general setting by \cite[Proposition 2.4]{KKLV}, the special behavior of line bundles on rational homogeneous varieties (see Section \ref{ssec:notation}) allows us to describe easily the associated weights.

In fact, given a line bundle $L$ on $G/B$ associated to a weight $\lambda\in\Mo(H')$, its description as a homogeneous bundle allows us to define a linearization of the action of $H'$ on $L$ for any character $m\in \Mo(H')$, as follows:
$$
h'[(g',v)]=[(h'g',h'^{m}v)], \mbox{ for any }h'\in H',[(g',v)]\in L.
$$ 
The important point to note here is that for the particular choice $m=\lambda$ the action of $H'$ on $L$ descends to an action of $H$. In fact, taking into account that the kernel of the map $H'\to H$ is the center of $G'$, we may write, for every $z\in\ker(H'\to H)$:
$$
z[(g',v)]=[(zg',z^{\lambda}v)]=[(g'z,z^{\lambda}v)]=[(g',v)], 
$$
for every $[(g',v)]\in L$. We may now compute the weights of the $H$-action with respect to $L$.

\begin{proposition}\label{prop:weights}
Let $L\in\Pic(G/B)$ be a line bundle associated to a weight $\lambda$. There exists a linearization of the $H$-action on $G/B$ whose weights are:
$$
\mu^H_L(wB)=\lambda-w(\lambda)\in \Mo(H), \mbox{ for any }w\in W.
$$
\end{proposition}

\begin{proof}
For every $h\in H$, $w\in W$, $v\in \C$, we have
\begin{multline*}
h[(w,v)]=[(hw,h^{\lambda}v)]=[(ww^{-1}hw,h^{\lambda}v)]=\\=[(w,(w^{-1}hw)^{-\lambda} h^{\lambda}v)]=
[(w,h^{-w(\lambda)+\lambda}v)].
\end{multline*}
\end{proof}

If $\C^*\subset H$ is a subtorus associated to a co-character $\sigma_i:\Mo(H)\to\Z$, then one knows (cf. \cite[Lemma 2.1 (ii)]{WORS2}) that every $\C^*$-fixed component contains at least an $H$-fixed point. In particular, the following Corollary describes the weights of the $\C^*$-action with respect to any line bundle $L$. 

\begin{corollary}\label{cor:weights}
Let $L\in\Pic(G/B)$ be a line bundle associated to a weight $\lambda$, and  $\C^*\subset H$ be a subtorus associated to a co-character $\sigma_i:\Mo(H)\to\Z$. There exists a linearization of the action whose weights on $H$-fixed points are:
$$
\mu_L(wB)=\sigma_i(\lambda-w(\lambda))\in \Mo(H), \mbox{ for any }w\in W.
$$
\end{corollary}

\subsubsection*{Fixed point components of $\C^*$-actions on rational homogeneous varieties}\label{sssec:fixedcomponents}

We already know that the fixed point components of the $\C^*$-action determined by a co-character $\sigma_i$ on a variety $G/P$ are $G^\perp$-homogeneous (see Proposition \ref{prop:fixcomp}). We will now describe more precisely these components. We start with the case of the complete flag manifold $G/B$. 

\begin{proposition}\label{prop:fixedareflags}
Let $G$ be a semisimple algebraic group, $B$ a Borel subgroup and $X=G/B$ the corresponding complete flag manifold. Consider the $\C^*$-action on $X$ induced by a co-character of the form $\sigma_i$ as above. Then the irreducible fixed point components of the action are flag manifolds with respect to the semisimple group $G^\perp$ transversal to the $\C^*$-action. 
\end{proposition}

\begin{proof}
By \cite[Lemma 2.1(ii)]{WORS2}, every irreducible component of $X^{\C^*}$ contains a point of $X^H$, hence, by Proposition \ref{prop:fixcomp}, every irreducible fixed point component of $X$ by $\C^*$  is of the form $G^{\perp}wB$ for some $w\in W$. Since the isotropy subgroup of $wB$ by the action of $G$ is $\conj_{w}(B)$, it follows that 
$$G^{\perp}wB=\dfrac{G^{\perp}}{G^{\perp}\cap \conj_{w}(B)}.$$
Since $G^{\perp}\cap \conj_{w}(B)$ is a Borel subgroup of $G^{\perp}$, the statement follows.
\end{proof}

As a direct consequence (by projecting the $\C^*$-fixed point components in $G/B$ via the natural projection to $G/P$), we obtain a description of the fixed point components of the $\C^*$-actions introduced above on rational homogeneous spaces of the form $G/P$. 

\begin{corollary}\label{cor:fixedareflags}
Let $G$ be a semisimple algebraic group, $B$ a Borel subgroup, $P\supset B$ a parabolic subgroup. 
Let $X=G/P$ be the corresponding rational homogeneous space, 
and consider the $\C^*$-action on $X$ induced by a co-character of the form $\sigma_i$ as above. Then for every $w\in W$ the irreducible fixed point component passing by the point $wP$ is the rational homogeneous variety:
$$
\dfrac{G^{\perp}}{G^{\perp}\cap \conj_{w}(P)}.
$$
\end{corollary}

\begin{remark}\label{rem:extreme}
It is worth observing that the extremal values of $\mu_L$ are achieved at the points $eB$ and $w_0B$, where $w_0\in W$ denotes the longest element of the Weyl group of $G$. In particular, the sink and the source of the $\C^*$-action induced by $\sigma_i$ are the $G^\perp$-orbits of $eB$, $w_0B$, respectively. More generally, if we consider any non-trivial $\C^*$-action induced by a co-character $\sigma:\Mo(H)\to\Z$, satisfying that $\sigma(\alpha_j)\geq 0$ for every $\alpha_j\in\Delta$, the description of the weights of the $H$-action on the tangent spaces of $G/B$ at $eB$ and $w_0B$  given in Proposition \ref{prop:Bruhat} tells us that the weights with respect to $\C^*$ at $T_{G/B,eB}$, $T_{G/B,w_0B}$ are all non-positive and non-negative, respectively. This implies that $eB,w_0B$ belong to the sink and the source of this action, respectively.
\end{remark}

\subsection{Equalized actions with isolated extremal fixed points on rational homogeneous spaces}\label{ssec:isolatedGP}

Later on we will study the birational maps associated to $\C^*$-actions on rational homogeneous varieties, and we will be interested in the case in which the domain of the birational map is the projective space. For that purpose we will concentrate on the $\C^*$-actions on rational homogeneous varieties that are equalized and have isolated sink and source. As we will see, these actions are completely classified.

\subsubsection*{Equalized actions on rational homogeneous spaces}\label{sssec:equalRH}

Via the correspondence of $\C^*$-actions on rational homogeneous varieties with co-characters and $\Z$-gradings (Section \ref{ssec:torRH}), it is known that equalized $\C^*$-actions correspond to {\em short gradings} (see  \cite{Fra1}), that is to those gradings for which $\fg_m=0$ if and only if $m\neq 0,\pm 1$. Throughout this Section we will always assume that $\fg$ is simple; in this case the possible short gradings of $\fg$ are known (cf. \cite[p.~42]{Tev2}): they correspond, up to conjugation, to some  co-characters of the form $\sigma_i$. The complete list of these co-characters can be read from the Table \ref{tab:short}. 

\renewcommand{\arraystretch}{1.1}
\begin{table}[h!!]
\caption{List of short gradings on simple Lie algebras.\label{tab:short}}
\begin{center}
\begin{tabular}{|C||C|C|C|C|C|C|}
\hline
\fg&\DA_n&\DB_n&\DC_n&\DD_n&\DE_6&\DE_7\\\hline
\sigma_i&\sigma_i\,\,(1\leq i\leq n)&\sigma_1&\sigma_n&\sigma_1,\sigma_{n-1},\sigma_n&\sigma_1,\sigma_6&\sigma_7\\\hline
\fg^\perp&\DA_{i-1}\times\DA_{n-i}&\DB_{n-1}&\DA_{n-1}&\DD_{n-1},\DA_{n-1},\DA_{n-1}&\DD_5,\DD_5&\DE_6\\\hline
\end{tabular}
\end{center}
\end{table}

\begin{remark}[Products]\label{rem:prods}
In the case in which $\fg$ is semisimple but not simple, a short grading on $\fg=\bigoplus_k\fg^k$ is given by the choice of a short grading on each of its simple direct summands $\fg^k$. If these short gradings are given by co-characters $\sigma_{i_k}$ in the list of Table \ref{tab:short}, then the short grading on $\fg=\bigoplus_k\fg^k$ is given by the co-character that we denote $\sum_k\sigma_{i_k}$. Translated into the language of $\C^*$-actions, an equalized $\C^*$-action on a rational homogeneous variety of the form $\prod_{k=1}^rG^k/P^k$, with $G^1,\dots,G^r$ simple, is given by the product of an equalized $\C^*$-action on every $G^k/P^k$. 
\end{remark}

\subsubsection*{Actions of $\C^*$ with isolated sink or source}\label{sssec:isolRH}

The following statement tells us when a $\C^*$-action associated to a co-character $\sigma_i$ on a rational homogeneous $G$-variety has isolated sink.

\begin{proposition}\label{prop:isofixed}
The $\C^*$-action  associated to the choice of a co-character $\sigma_i$ on a rational homogeneous $G$-variety $\cD(I)$ has isolated sink if and only if $I=\{i\}$.
\end{proposition}

\begin{proof}
Write $\cD(I)$ as $G/P$ for some $P\supset B$. By Remark \ref{rem:extreme}, the sink of the action will be the fixed point component passing by $eP$. Then, by Corollary \ref{cor:fixedareflags}, the action has isolated sink if and only if 
$G^\perp\subset P$, and this is only possible if $I=\{i\}$.
\end{proof}

\begin{remark}[Reversion]\label{rem:revert}
A similar argument allows us to state that our $\C^*$-action has isolated source if and only if $G^\perp\subset \conj_{w_0}P$. Now we use the description of $w_0$ that we have recalled in Remark \ref{rem:longest}. If the action of $w_0$ on $\Mo(H)$ equals $-\id$, then $ \conj_{w_0}P$ is nothing but the opposite parabolic subgroup of $P$, which contains $G^\perp$ if and only if $P$ does. This tells us that in the cases:
$$
\DB_n,\quad\DC_n,\quad\DD_n\,\,\mbox{($n$ even)},\quad \DE_7
$$
the action of $\C^*$ associated to $\sigma_i$ on $\cD(I)$ has isolated sink and source, if and only if $I=\{i\}$. 

In the remaining cases, we use Remark \ref{rem:longest} to identify the parabolic subgroup $\conj_{w_0}(P)$, and to check when it does contain $G^\perp$ obtaining:
\begin{itemize}
\item for every $i$, $\DA_n(I)$ has isolated sink if and only if $I=\{i\}$ and isolated source if and only if $I=\{n+1-i\}$; in particular, if $n$ is odd and $i=(n+1)/2$, $\DA_{n}(i)$ has isolated sink and source;
\item for $i=1$, $\DD_n(I)$ has isolated sink and source if and only if $I=\{1\}$; for $i=n-1,n$, the action on $\DD_n(I)$ has isolated sink if and only if $I=\{i\}$, and isolated source if and only if $I=\{n-1,n\}\setminus\{i\}$;
\item for $i=1,6$, the action on $\DE_6(I)$ has isolated sink if and only if $I=\{i\}$, and isolated source if and only if $I=\{1,6\}\setminus\{i\}$.
\end{itemize}

\end{remark}

\begin{remark}[Fixed points on products]\label{rem:isofixed2}
Note that, in the setup of Proposition \ref{prop:isofixed}, even if $\fg$ not simple, a rational homogeneous variety of the form $\cD(i)$ is always a quotient of a semisimple algebraic group with simple Lie algebra; more concretely, if we denote by $\cD'\subset \cD$ the connected component of $\cD$ containing the node $i$, we have $\cD(i)=\cD'(i)$.  
Then, if $\fg$ is a direct sum $\sum_{k=1}^r\fg^k$ of simple Lie algebras, and we choose a node $i_k$ on the Dynkin diagram $\cD_k$ of $\fg^k$, for every $k$, the $\C^*$-action on the variety $\cD(i_1,\dots,i_r)=\cD_1(i_1)\times\dots\times \cD_r(i_r)$ given by the co-character $\sum_{k}\sigma_{i_k}$ is the product of the actions of $\C^*$ on the $\cD_r(i_r)$'s given by the co-characters $\sigma_{i_k}$. Then the fixed point components of this action will be the products of the fixed point components of the $\C^*$-actions on each factor and, in particular, the sink of the action will be isolated. For instance, the action induced by the character $\sigma_1+\sigma_1$ on $\DA_n(1)\times \DA_m(1)$ has four fixed components, isomorphic to a point (the sink), $\DA_{n-1}(1)$, $\DA_{m-1}(1)$, and $\DA_{n-1}(1)\times \DA_{m-1}(1)$.
\end{remark}

Proposition \ref{prop:isofixed} provides a list of rational homogeneous varieties  of Picard number one admitting equalized $\C^*$-actions with isolated sink. Following Remarks \ref{rem:prods}, \ref{rem:isofixed2}, actions of this kind on  rational homogeneous varieties of arbitrary Picard number are obtained by considering products of these.  

In order to study the remaining fixed point components of each of these actions, one may proceed as follows. First of all one considers the natural embedding of $\cD(i)$ on the irreducible representation $V_{\lambda_i}$ of $\fg$ associated to the fundamental weight $\lambda_i\in\Mo(H)$. Then one considers the short grading of $\fg$ determined by the root $\alpha_i$, splits $V_{\lambda_i}$ as a direct sum of irreducible $\fg_0$-modules, and studies the intersection of $\cD(i)$ with the corresponding projectivizations. The outcome of the process  is the following:

\begin{proposition}\label{prop:equalisol}
Let $G$ be a semisimple group, with simple Lie algebra. The complete list of rational homogeneous $G$-varieties 
admitting equalized $\C^*$-actions with an extremal isolated fixed point, and of their irreducible fixed point components, is given in Table \ref{tab:equalisol}.
\begin{table}[!h!!]\caption{Equalized $\C^*$-actions with isolated extremal fixed points on rational homogeneous spaces of Picard number one.
\label{tab:equalisol}}
\begin{center}
\begin{tabular}{|c|C|C|C|}
\hline
$\fg$&\cD(i)&\mbox{weights } j&\mbox{fixed point comp. of weight $j$}\\\hline\hline
\multirow{3}{*}{$\DA_n$}&\DA_n(i)&0,\dots,\min\{i,n-i+1\}&\DA_{i-1}(i-j)\times \DA_{n-i}(j)\\\cline{2-4}
&\multirow{2}{*}{$\DA_n(n+1-i)$}&\max\{0,n-2i+1\},\dots&\multirow{2}{*}{$\DA_{n-i}(j)\times \DA_{i-1}(n+1-i-j)$}\\
& &\dots  , n-i+1&\\\hline
$\DC_n$&\DC_n(n)&0,\dots,n&\DA_{n-1}(j)\\\hline
$\DB_n$&\DB_n(1)&0,1,2&\DB_{n-1}(0),\DB_{n-1}(1),\DB_{n-1}(0)\\\hline
$\DD_n$&\DD_n(1)&0,1,2&\DD_{n-1}(0),\DD_{n-1}(1),\DD_{n-1}(0)\\\hline
\multirow{2}{*}{$\DD_n$}&\DD_n(n)&0,\dots,\lfloor n/2\rfloor&\DA_{n-1}(n-2j)\\\cline{2-4}
&\DD_n(n-1)&0,\dots,\lfloor n/2\rfloor&\DA_{n-1}(n-2(\lfloor n/2\rfloor-j))\\\hline
\multirow{2}{*}{$\DE_6$}&\DE_6(1)&0,1,2&\DD_5(6),\DD_5(5),\DD_5(1)\\\cline{2-4}
&\DE_6(6)&0,1,2&\DD_5(1),\DD_5(5),\DD_5(6)\\\hline
$\DE_7$&\DE_7(7)&0,1,2,3&\DE_6(0),\DE_6(1),\DE_6(6),\DE_6(7)\\\hline
\end{tabular}
\end{center}
\end{table}
\end{proposition}

\begin{remark}\label{rem:equalisol}
Table \ref{tab:equalisol} above must be read with the following conventions: first of all, for any $k$, we set $\cD(k+1)=\cD(0)$ to be a point for any diagram with $k$ nodes $\cD$. Second, the weights of the fixed point components are taken with respect to a linearization of the ample generator of the Picard group of the corresponding variety, set to have weight $0$ at the sink of the action. In this way we can see that the criticality and the bandwidth of the action coincide with the maximal weight $\delta$ of the action. 
\end{remark}

\subsubsection*{$\C^*$-actions with isolated sink and source}\label{sssec:isol2RH}

Finally, as a consequence of Proposition \ref{prop:equalisol}, we get a description of all the actions with isolated extremal fixed points, which will have Cremona transformations as their induced birational maps.

\begin{corollary}\label{cor:equalisol}
The complete list of rational homogeneous spaces of Picard number one admitting equalized $\C^*$-actions with isolated sink and source is given in Table \ref{tab:equalisol2}.
\renewcommand{\arraystretch}{1.2}
\begin{table}[!h!!]
\caption{Equalized $\C^*$-actions with isolated sink and source.\label{tab:equalisol2}}
\begin{center}
\begin{tabular}{|r||C|C|C|C|C|C|}
\hline
variety &\DA_{2n-1}(n)&\DC_n(n)&\DB_n(1)&\DD_n(1)&\DD_n(n)&\DE_7(7)\\\hline
restrictions &n\geq 1&n\geq 2&n\geq 2&n\geq 4&n\geq 6, \,\,n\, \mbox{even}&\\\hline
dimension&n^2&{n+1}\choose{2}&2n-1&2n-2&{n}\choose{2}&27\\\hline
$\delta$&n&n&2&2&n/2&3\\\hline
\end{tabular}
\end{center}
\end{table}
\end{corollary}

\begin{remark}\label{rem:bispecial}
In the list above only four cases are geometric realizations of bispecial Cremona transformations (see \cite[Definition~4.1]{WORS4a}); they are exactly the varieties classified in \cite[Theorem~8.10]{WORS1}, which correspond to the whole list of quadro-quadric special Cremona transformations of Ein and Shepherd-Barron, associated to the four Severi varieties:
$$
\DA_{5}(3),\quad\DC_3(3),\quad\DD_6(6),\quad \DE_7(7).
$$
Note also that the list of quadro-quadric Cremona transformations with smooth fundamental locus (classified in \cite[Proposition~5.6]{PiRu}) is completed  by adding the birational map induced by the equalized $\C^*$-action on $\P^1\times Q^n$, where $Q^n$ denotes the $n$-dimensional smooth quadric. From the description of the fixed point components of the action, it follows that the fundamental locus of the associated birational map in this case is the union of a quadric $Q'$ of dimension $n-2$ and a point outside of the linear span of $Q'$. 
\end{remark}

%% file: DiagE6num.tex
\ifx\du\undefined
  \newlength{\du}
\fi
\setlength{\du}{3.3\unitlength}
\begin{tikzpicture}
\pgftransformxscale{1.000000}
\pgftransformyscale{1.000000}

\definecolor{dialinecolor}{rgb}{0.000000, 0.000000, 0.000000} 
\pgfsetstrokecolor{dialinecolor}
\definecolor{dialinecolor}{rgb}{0.000000, 0.000000, 0.000000} 
\pgfsetfillcolor{dialinecolor}


\pgfsetlinewidth{0.300000\du}
\pgfsetdash{}{0pt}
\pgfsetdash{}{0pt}

\pgfpathellipse{\pgfpoint{-6\du}{0\du}}{\pgfpoint{1\du}{0\du}}{\pgfpoint{0\du}{1\du}}
\pgfusepath{stroke}
\node at (-6\du,0\du){};

\pgfpathellipse{\pgfpoint{14\du}{-7\du}}{\pgfpoint{1\du}{0\du}}{\pgfpoint{0\du}{1\du}}
\pgfusepath{stroke}
\node at (14\du,-7\du){};

\pgfpathellipse{\pgfpoint{4\du}{0\du}}{\pgfpoint{1\du}{0\du}}{\pgfpoint{0\du}{1\du}}
\pgfusepath{stroke}
\node at (4\du,0\du){};

\pgfpathellipse{\pgfpoint{14\du}{0\du}}{\pgfpoint{1\du}{0\du}}{\pgfpoint{0\du}{1\du}}
\pgfusepath{stroke}
\node at (14\du,0\du){};

\pgfpathellipse{\pgfpoint{24\du}{0\du}}{\pgfpoint{1\du}{0\du}}{\pgfpoint{0\du}{1\du}}
\pgfusepath{stroke}
\node at (24\du,0\du){};

\pgfpathellipse{\pgfpoint{34\du}{0\du}}{\pgfpoint{1\du}{0\du}}{\pgfpoint{0\du}{1\du}}
\pgfusepath{stroke}
\node at (34\du,0\du){};

\pgfsetlinewidth{0.300000\du}
\pgfsetdash{}{0pt}
\pgfsetdash{}{0pt}
\pgfsetbuttcap

{\draw (-5\du,0\du)--(3\du,0\du);}
{\draw (5\du,0\du)--(13\du,0\du);}
{\draw (15\du,0\du)--(23\du,0\du);}
{\draw (25\du,0\du)--(33\du,0\du);}
{\draw (14\du,-1\du)--(14\du,-6\du);}

\node[anchor=west] at (38\du,0\du){${\rm E}_6$};

\node[anchor=south] at (-6\du,1.1\du){$\scriptstyle 1$};

\node[anchor=south] at (4\du,1.1\du){$\scriptstyle 3$};

\node[anchor=south] at (14\du,1.1\du){$\scriptstyle 4$};

\node[anchor=south] at (24\du,1.1\du){$\scriptstyle 5$};

\node[anchor=south] at (34\du,1.1\du){$\scriptstyle 6$};

\node[anchor=east] at (12.9\du,-7\du){$\scriptstyle 2$};

\end{tikzpicture} 

%% file: cremona.tex

\section{Equivariant Cremona transformations}\label{sec:examples2}

This Section is devoted to the Cremona transformations induced by the equalized $\C^*$-actions with isolated sink and source on rational homogeneous varieties (Corollary \ref{cor:equalisol}). An important property satisfied by these transformations is that they commute with the action of a semisimple group: the transversal group $G^\perp$. As we will see, under this representation-theoretical point of view one may interpret these Cremona transformations as maps of inversion. 

A first important observation is the following:

\begin{corollary}\label{cor:homocremona}
Let $X=G/P$ be a rational homogeneous variety endowed with the $\C^*$-action associated to a co-character $\sigma_i:\Mo(H)\to \Z$, and let $G^\perp$ be the corresponding transversal group. Denote by $Y_\pm$ the sink and source of the $\C^*$-action, and by $\psi:\P(\cN_{Y_-|X}^\vee)\dashrightarrow \P(\cN_{Y_+|X}^\vee)$ the induced birational transformation. Then the action of $G^\perp$ on $X$ induces actions on $\P(\cN_{Y_\pm|X}^\vee)$ so that $\psi$ is $G^\perp$-equivariant. 
\end{corollary}

\begin{proof}
As we noted in Remark \ref{rem:commute} the action of $G^\perp$ on $X$ leaves invariant $Y_\pm$, hence it extends to an action on $\P(\cN_{Y_\pm|X}^\vee)$. Moreover, since $\P(\cN_{Y_\pm|X}^\vee)$ can be thought of as the geometric quotient parametrizing $1$-dimensional $\C^*$-orbits converging to $Y_\pm$, and $G^\perp$ sends closures of $\C^*$-orbits to closures of $\C^*$-orbits (Remark \ref{rem:commute}), then $\psi$ is $G^\perp$-equivariant. 
\end{proof}

Applying \ref{sssec:induced} to each of the actions described in Table \ref{tab:equalisol} --in which the source is a point-- we get a birational transformation $\psi$ onto the projectivization of the tangent space of $X$ at the point $Y_+$. When the sink $Y_-$ is positive dimensional, $\P(\cN^\vee_{Y_-|X})$ is the projectivization of a homogeneous bundle over $Y_-$. In particular, with the exception of the case in which $X=\P^n$ and the sink of the action is a hyperplane, the variety $\P(\cN^\vee_{Y_-|X})$ has Picard number bigger than one. On the other hand, in the cases of Table \ref{tab:equalisol2}, we obtain Cremona transformations. 

\begin{definition}\label{def:Cremonaequiv}
Let $G$ be a semisimple group, and consider two projective representations $\P(V_1)$, $\P(V_2)$ of $G$ of the same dimension. A Cremona transformation $\psi:\P(V_1)\to \P(V_2)$ is said to be {\em $G$-equivariant} if it commutes with the action of $G$.
\end{definition}

The most obvious examples of equivariant Cremona transformations are those listed in Table \ref{tab:equalisol2}, that we will now describe geometrically. Let us also pose the following  problem that arises naturally at this point: 

\begin{problem}\label{prob:equivariant}
Classify $G$-equivariant Cremona transformations among two irreducible projective representations of a semisimple group $G$, constructing geometric realizations for them.
\end{problem}

\subsection{Example: Quadrics}\label{ssec:quadrics}

The least interesting of the birational transformations induced by the actions of Table \ref{tab:equalisol2} is the case of the non-singular quadrics $Q$ of the form $\DB_n(1)$ or $\DD_n(1)$; following \cite[Section 5.2]{WORS3}, the fact that the criticality of the action is two, and that the extremal fixed points are isolated, forces $\psi$ to be an isomorphism. 

This can be also seen as follows: the $\C^*$-action on $Q$ has two extremal fixed points $P_-$, $P_+$, which are not linked by a line in $Q$. Every tangent direction $v_-\in \P(T_{Q,P_-}^\vee)$ determines a plane $\pi_{v_-}$ containing $P_-,P_+$ and the direction $v_-$. This plane meets $Q$ along a conic $C$ passing through $P_-,P_+$; this conic is reduced, not necessarily irreducible, and it is non-singular at $P_-$, $P_+$. Moreover, since the action is equalized, the plane $\pi_{v_-}$ is $\C^*$-invariant, hence $C$ is $\C^*$-invariant and we may conclude that $\psi(v_-)=[T_{C,P_+}]\in \P(T_{Q,P_+}^\vee)$. In particular $\psi$ is well-defined for every $v_-\in \P(T_{Q,P_-}^\vee)$, so it is regular. The same clearly holds for its inverse.

\subsection{Example: Balanced Grassmannians}\label{ssec:grass}

Let us consider two complex vector spaces $V_-,V_+$ of dimension $n\geq 1$, and the action of $\C^*$ on $V_-\oplus V_+$ that leaves $V_-$, $V_+$ invariant,  whose weights on these spaces are $0$ and $1$, respectively. Then we consider the induced action on the Grassmannian $\DA_{2n-1}(n)$ of $n$-dimensional linear subspaces of $V_-\oplus V_+$. Its extremal fixed points correspond to the subspaces $V_-$ and $V_+$. 
By \ref{sssec:induced}, the action of $\C^*$ induces a Cremona transformation among the projectivizations of the tangent spaces of $\DA_{2n-1}(n)$ at the extremal fixed point components:
$$
\psi:\P(V_-^\vee\otimes V_+)\dashrightarrow \P(V_+^\vee\otimes V_-).
$$
In order to describe this map, we note that a general element of $\DA_{2n-1}(n)$ is determined by a skew-symmetric tensor:
$$
s:=(v_1^-+v_1^+)\wedge \dots \wedge (v_n^-+v_n^+),
$$
where $v^-_1,\dots, v^-_n\in V_-$ and  $v^+_1,\dots, v^+_n\in V_+$ are linearly independent vectors.
Since the orbit of $[s]$ in $\DA_{2n-1}(n)$ is the set of points of the form:
$$
t[s]=\left[(v_1^-\wedge\dots \wedge v_n^-)+ t\sum_{i=1}^n (v_1^-\wedge\dots \wedge v_i^+ \wedge\dots \wedge v_n^-)+O(t^2)\right],
$$

then the limiting point of the orbit of $[s]$ in the blowup of $\DA_{2n-1}(n)$ along the source and the sink when $t$ goes to $0$ is:
$$
\lim_{t\to 0}t[s]=\left[\sum_{i=1}^n v_1^-\wedge\dots \wedge v_i^+ \wedge\dots \wedge v_n^-\right],
$$
and analogously:
$$
\lim_{t\to \infty}t[s]=\left[\sum_{i=1}^n v_1^+\wedge\dots \wedge v_i^- \wedge\dots \wedge v_n^+\right].
$$
These points can be identified, respectively, with the class of the linear map from $V_-$ to $V_+$ sending every $v_i^-$ to $v_i^+$, and its inverse. In other words, as noted by Thaddeus in \cite[Section 4]{Thaddeus}, the map $\psi$ is the projectivization of the inversion map. 

\begin{remark}\label{rem:inversion}
Note that, in the particular case $n=3$, the exceptional locus of $\psi$ is a Segre variety $\P^2\times \P^2\subset \P^8$, which is the Severi variety associated to the division algebra $\C$. The map $\psi$ is one of the special Cremona transformations characterized in \cite{ESB}. 
\end{remark}

As we noted in the Introduction, the standard Cremona transformation from $\P^{n-1}$ to $\P^{n-1}$ is a restriction of the projectivization of the inversion map of $n\times n$ matrices. Moreover, its geometric realization given by the diagonal natural $\C^*$-action on $(\P^1)^n$ can be $\C^*$-equivariantly embedded into $\DA_{2n-1}(n)$. We would like to present here the problem of classifying the Cremona transformations that satisfy this property. Philosophically speaking, one would like to understand how far is the inverse map from being a ``universal'' Cremona transformation.

\begin{problem}\label{prob:immersion}
Classify Cremona transformations that are restrictions of the inversion map (Example \ref{ssec:grass}), studying when their geometric realizations may be embedded $\C^*$-equivariantly into $\DA_{2n-1}(n)$. 
\end{problem}

\subsection{Example: Isotropic Grassmannians}\label{ssec:isgrass}

We refer to \cite{MMW} for more details on this example and its applications. The equalized $\C^*$-action on $\DC_n(n)$ described in Table \ref{tab:equalisol2} can be seen as a restriction of the one we have just studied in the previous Example. In order to see this we start from a direct sum $V_-\oplus V_+$ as above, and consider an isomorphism $\phi:V_-\to V_+^\vee$, which defines a non-degenerate skew-symmetric $2$-form $\omega$ on $V_-\oplus V_+$:
$$
\omega(u^-+u^+,v^-+v^+):=\phi(u^-)(v^+)-\phi^T(u^+)(v^-).
$$
The $\C^*$-action on $V_-\oplus V_+$ of the previous example, which acts trivially on $V_-$ and homothetically on $V_+$, preserves orthogonality with respect to $\omega$, hence leaves invariant the subset $\DC_n(n)\subset \DA_{2n-1}(n)$ of $\omega$-isotropic subspaces. Since $[V_-],[V_+]\in \DC_n(n)$, they are the extremal isolated fixed points of the action. 

Note that the tangent space of $\DC_n(n)$ at $[V_\pm]$ is isomorphic to $S^2V_{\pm}^\vee$, and that its embedding into $T_{\DA_{2n-1}(n),[V_\pm]}\simeq V_\pm^\vee\otimes V_\mp$ is induced naturally by the isomorphism $\phi$. Then the birational map 
$$\psi:\P(S^2V_-^\vee)\dashrightarrow \P(S^2V_+^\vee)\simeq \P(S^2V_-)$$
is just the (projectivization of the) inversion map, restricted to symmetric tensors.

\begin{remark}\label{rem:isgrass}
In the particular case $n=3$, we obtain another special Cremona transformation: the one whose exceptional locus is  the Veronese surface in $\P^5$, which is the Severi variety associated to the division algebra $\R$.
\end{remark}
 
 \subsection{Example: Spinor varieties}\label{ssec:spin}

The case of $\DD_n(n)$ ($n$ even) is similar to the one of $\DC_n(n)$ described above, in the sense that it can also be obtained as a restriction of the case of $\DA_{2n-1}(n)$  by considering a non-degenerate symmetric $2$-form on $V_-\oplus V_+$:
$$
q(u^-+u^+,v^-+v^+)=\phi(u^-)(v^+)+\phi^T(u^+)(v^-),
$$
induced by a given isomorphism $\phi:V_-\to V_+^\vee$.  
This defines a quadric $Q\subset \P(V_-^\vee\oplus V_+^\vee)$ containing $\P(V_-^\vee), \P(V_+^\vee)$, which are invariant by the action of $\C^*$, 
 and so we get a $\C^*$-action on $Q$ with extremal fixed point components $\P(V_-^\vee), \P(V_+^\vee)$. As in the case of $\DC_n(n)$, this induces a $\C^*$-action on the set of linear subspaces contained in $Q$ of maximal dimension, which consists of two (isomorphic) irreducible components, $\DD_n(n-1)$, $\DD_n(n)$. 
 Since the dimension of the intersection of two maximal linear subspaces in $Q$ is congruent with $(n-1)$ modulo $2$ if and only if they belong to the same irreducible family (\cite[Theorem~21.14]{Harr}), it follows that $\P(V_-^\vee), \P(V_+^\vee)$ belong to the same family if and only if $n$ is even. Note that in Table \ref{tab:equalisol2} we have excluded the case $n=4$ since in this case $\DD_4(4)\simeq \DD_4(1)$, and so the action in question has been already described in \ref{ssec:quadrics} .

The interpretation of the birational map $\psi$ in this case is analogous to the case of the isotropic Grassmannian. This time we have isomorphisms:
$$
T_{\DD_{n}(n),[V_\pm]}\simeq \bigwedge^2 V_\pm^\vee,
$$
and we may think of
$$\psi:\P\left(\bigwedge^2V_-^\vee\right)\dashrightarrow \P\left(\bigwedge^2V_+^\vee\right)\simeq \P\left(\bigwedge^2V_-\right)$$
as the projectivization of the inversion of skew-symmetric tensors.

\begin{remark}\label{rem:spin}
As in the previous examples, in one particular case ($n=6$) we obtain a special Cremona transformation; in this case  $\psi$ is the quadro-quadric transformation whose exceptional locus is the Grassmannian $\DA_5(2)$, which is the Severi variety associated to the division algebra $\H$ of quaternions.   
\end{remark}

\subsection{Example: The $E_7$ case}\label{ssec:E7}

The last special Cremona transformation, associated to the algebra $\Oc$ of octonions, makes its appearance in the case of the $27$-dimensional variety $\DE_7(7)$, which is the unique rational homogeneous variety of exceptional type admitting an equalized $\C^*$-action with isolated extremal fixed points. First of all, a straightforward computation shows that the bandwidth of the action with respect to the ample generator of the Picard group is three. 
The inner fixed point components of the action are isomorphic to the variety of ($\DE_7$-isotropic) lines passing by the sink and the source, which is isomorphic to the Cartan variety $\DE_6(1)$, which is the Severi variety associated to the algebra $\Oc$ of octonions. Then $\DE_6(1)$ is the exceptional locus of the Cremona transformation $\psi:\P^{26}\dashrightarrow \P^{26}$. Since moreover one can show that $\psi$ is quadro-quadric (see \cite[Section 8]{WORS1}), we conclude that this is precisely the special Cremona transformation associated to that Severi variety.

%% file: RHbundles.tex

\section{Geometric realizations and rational homogeneous bundles}\label{sec:derived}

By restricting $\C^*$-actions on rational homogeneous varieties to invariant subvarieties we may construct more examples of $\C^*$-actions, that will be geometric realizations of other birational transformations. We will illustrate this procedure here by considering restrictions of $\C^*$-actions to some invariant rational homogeneous bundles over $\P^1$, showing in particular that the homogeneous cases of the list of special transformations of type $(2,1)$, classified by Fu and Hwang (cf. \cite{FuHw}), can be obtained in this way. Let us start by describing the varieties and $\C^*$-actions that we will use.

\begin{setup}\label{setup:derived}
We start with the adjoint group $G$ of a simple Lie algebra $\fg$ with Dynkin diagram $\cD$, a Cartan and a Borel subgroup $H\subset B \subset G$, and a co-character $\sigma_i:\Mo(H)\to \Z$, $i\in D$, inducing a short grading of $\fg$ (see Table \ref{tab:short}). We then have an equalized $\C^*$-action on every rational homogeneous variety of the form $\cD(I)$, that has isolated sink if $I=\{i\}$.  
We will denote by $P_i$ the parabolic subgroup containing $B$ such that $\cD(i)=G/P_i$. The transversal subgroup of the action (which is the semisimple part of $P_i$) will be denoted by $G^\perp\subset G$; its Dynkin diagram $\cD_i$ is the result of cancelling the node $i$ on $\cD$. As usual, we denote by $L_i$ the ample generator of $\Pic(\cD(i))$; the weights of the action with respect to $L_i$ at the sink and the source will be denoted by $0$, $\delta$, respectively. Given a $\C^*$-invariant irreducible curve $\ell\subset \cD(i)$, and a subset $\emptyset \neq J\subset D\setminus \{i\}$, the actions of $\C^*$ on $\cD(D)$ and $\cD(J\cup \{i\})$ restrict to the fiber products
$$
\ol{X}:=\cD(D)\times_{\cD(i)}\ell\subset\cD(D), \quad \ol{X}(J):=\cD(J\cup \{i\})\times_{\cD(i)}\ell\subset\cD(J\cup \{i\}),
$$
so that the natural maps
$$
\ol{X}\lra \ol{X}(J) \lra \ell
$$
are $\C^*$-equivariant. 
For simplicity, we will consider only the case in which $\ell$ is a line passing by the sink, joining $Y_-$ to a component of weight $1$. More concretely, we consider the fixed point $r_iP_i\in G/P_i$; since the images of $eB$ and $r_iB$ via the $i$-th $\P^1$-bundle structure of $G/B$  are the same (because $r_i\in BW(i)B$), it follows that they are contained in a curve in the class $\Gamma_i$, whose image into $\cD(i)$ will be our chosen curve $\ell$. The situation has been represented in Figure \ref{fig:derived} below.
\end{setup}

\begin{figure}[!h!]
\includegraphics[width=6cm]{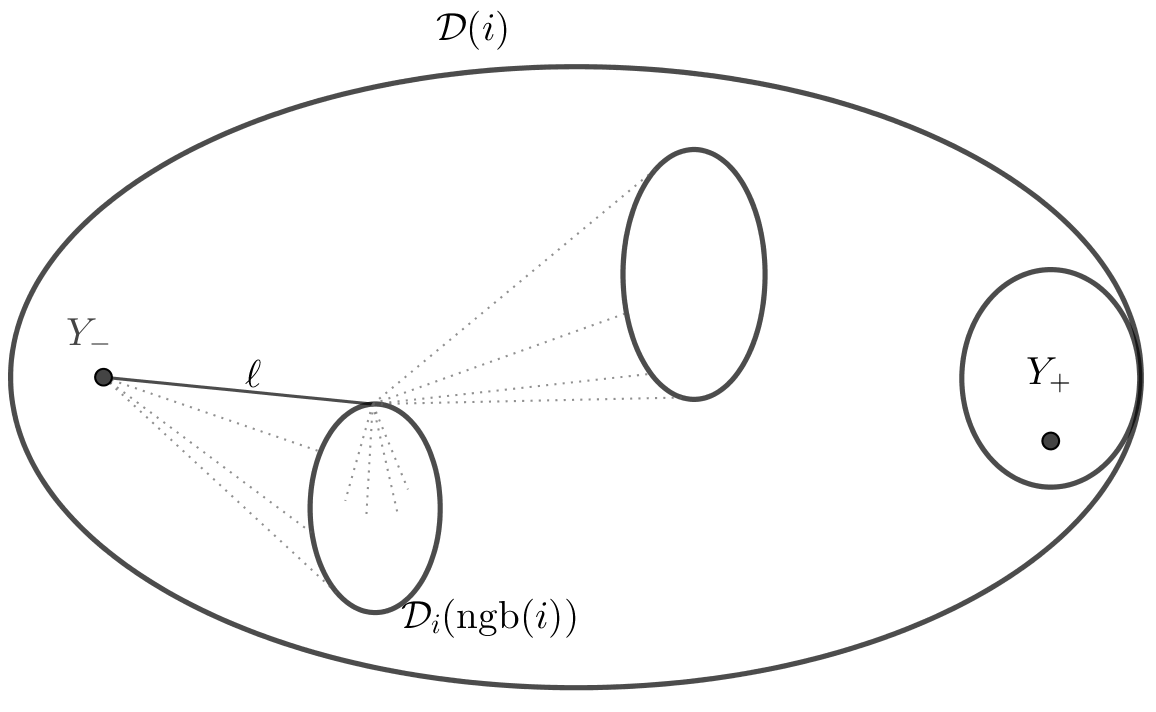}\qquad \includegraphics[width=5.5cm]{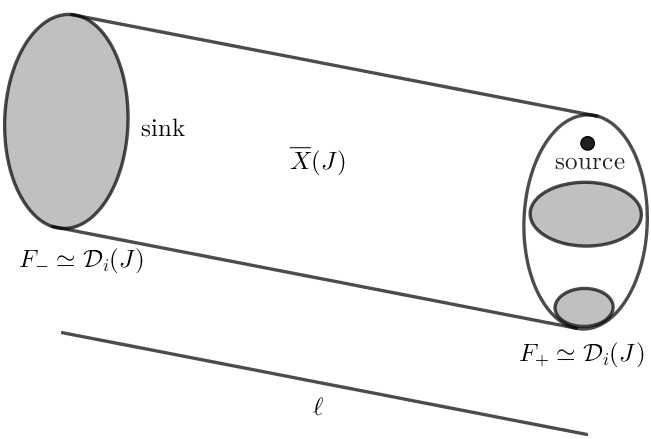}
\caption{A rational homogeneous variety with isolated sink, and the induced action on $\ol{X}(J)$.\label{fig:derived}}
\end{figure}

\begin{remark}\label{rem:invcurve}
The above construction of a $\C^*$-invariant smooth subvariety works in a much broader situation. Our particular choice of $\ell$ and $\ol{X}(J)$ is motivated by the fact that we want its induced birational transformation to go from a rational homogeneous space of the form $\cD_i(J)$ to a projective space. This is guaranteed as follows: first, as we will see below, assuming that the $\C^*$-invariant $\ell$ passes by $Y_-$ suffices to assure that the sink of the action on $\ol{X}(J)$ is a variety of the form $\cD_i(J)$, for every $J$. Second, for our choice of $\ell$ it will always be possible to choose $J$ so that $\ol{X}(J)$ has isolated source.  
\end{remark}

\subsection{Special birational transformations of type $(2,1)$}\label{ssec:FuHw}

In order to describe the birational transformations whose geometric realizations are the varieties $\ol{X}(J)$ described above, let us start by determining the fixed locus of the action of $\C^*$ on $\ol{X}$. Since the projection to $\ell$ is $\C^*$-equivariant, we have a surjective map from $\ol{X}\,\!^{\C^*}$ onto $\ell^{\,\C^*}$, which consists of two points: $y_-:=eP_i$ (so that $Y_-=\{y_-\}$ is the sink of the action in $\cD(i)$), and $y_+:=r_iP_i$ (of $L_i$-weight equal to $1$). Let us denote by $F_-,F_+\subset \ol{X}$ the corresponding fibers over $y_-,y_+$, which are isomorphic to $P_i/B\simeq G^\perp/B^\perp$. 

By Corollary \ref{cor:fixedareflags}, the fiber $F_-$ is a fixed point component of the action; on the other hand $F_+$ is $\C^*$-invariant, but the action of $\C^*$ on $F_+$ is not trivial:

\begin{lemma}\label{lem:actionF-}
In the situation of Setup \ref{setup:derived}, the induced action of $\C^*$ on $F_+$ is given by a co-character $\sigma_+$ inducing a short grading of the Lie algebra $\fg^\perp$. In particular, we may choose $J\subset D\setminus\{i\}$ so that the induced $\C^*$-action on $F^+$ has isolated source. 
With the standard numbering of the nodes in (each of the connected components of) the Dynkin diagram $\cD_i$ of $G^\perp$, the values of $\sigma_+\in\Mo(H^\perp)^\vee$ and $J$ have been described in Table \ref{tab:inducedF}.
\begin{table}[!h!!]\caption{Induced $\C^*$-action of the fiber $F_+$.
\label{tab:inducedF}}
\begin{center}
\begin{tabular}{|C|C|C|C|C|}
\hline
\cD&\cD(i) &\cD_i&\sigma_+&J\\\hline\hline
\DA_n&\DA_n(i) &\DA_{i-1}\sqcup \DA_{n-i} &\sigma_{i-1}+\sigma_{1}&\{1,n\}\\\hline
\DC_n&\DC_n(n) &\DA_{n-1}&\sigma_{n-1}&\{1\}\\\hline
\DB_n&\DB_n(1)&\DB_{n-1}&\sigma_{1}&\{2\}\\\hline
\DD_n&\DD_n(1)&\DD_{n-1}
&\sigma_{1}&\{2\}\\\hline
\DD_n&\DD_n(n)&\DA_{n-1}&\sigma_{n-2}&\{2\} \\\hline
\DE_6&\DE_6(1)&\DD_5&\sigma_5&\{2\}\\\hline
\DE_7&\DE_7(7)&\DE_6&\sigma_6&\{1\} \\\hline
\end{tabular}
\end{center}
\end{table}
\end{lemma}

\begin{proof}
The fiber $F_+$ over $y_+=r_iP_i$ is a quotient of the subgroup $\conj_{r_i}(G^\perp)\subset G$ by the Borel subgroup $\conj_{r_i}(B^\perp)$. In order to study the action of $\C^*$ on $F_+$ we consider the Cartan subgroup $H^\perp:=G^\perp\cap H$, whose character lattice is $\Mo(H^\perp)=\ker(\sigma_i)\subset \Mo(H)$. Then the induced action of $\C^*$ on $F_+$ is given by the co-character $\sigma_+$ obtained via the composition:
$$
\Mo(\conj_{r_i}(H^\perp))=r_i
(\Mo(H^\perp))\hookrightarrow \Mo(H)\stackrel{\sigma_i}{\lra}\Z.
$$

The image into $\Mo(H)$ of the base of positive simple roots associated to the choice of the Borel subgroup $\conj_{r_i}(B^\perp)\subset \conj_{r_i}(G^\perp)$ (which is $\conj_{r_i}(\Delta\setminus\{\alpha_i)$) consists of all the roots of the form $r_i(\alpha_j)$, $j\neq i$. In all the cases described in Table \ref{tab:short} we have that 
$$r_i(\alpha_j)=\begin{cases}\alpha_j+\alpha_i&\text{if $j\in\nbh{i}$,}\\\alpha_j &\text{otherwise,}\end{cases}$$
for $j\neq i$, so that the co-character $\sigma_+$ sends every positive simple root $\alpha_j$ to $1$ if $j$ neighbors $i$, and to zero otherwise.  
Finally a case by case argument provides the list of co-characters in Table \ref{tab:inducedF}. The value of the set $J$ follows then by  Proposition \ref{prop:equalisol}.
\end{proof}

\begin{setup}\label{setup:derived2}
Let us set, throughout the rest of Section \ref{sec:derived},
$$
X:=\ol{X}(J),
$$
for each one of the choices of $\cD$, $i$ and $J$ as above. As we already noted, with this choice the sink and the source of the $\C^*$-action on $X$ 
are, respectively:
\begin{itemize}
\item[($-$)] the rational homogeneous variety $\cD_i(J)$, and
\item[($+$)] an isolated point.
\end{itemize}
\end{setup}

By applying now Corollary \ref{cor:homocremona}, we immediately get:

\begin{corollary}\label{cor:actionF+}
In the setting of Setup \ref{setup:derived}, being $J$ as in Table \ref{tab:inducedF}, the induced $\C^*$-action on $X$  
defines a birational transformation:
$$
\xymatrix{\cD_i(J)\ar@{-->}[r]^(.40){\psi}&\P^{\dim(\cD_i(J))}}.
$$
\end{corollary}

Let us describe some properties of the $\C^*$-action  on $X$,  
and of its associated  birational transformation. We start by listing in Table \ref{tab:inducedF2} below the fixed point components of the action, that can be obtained out of Lemma \ref{lem:actionF-} and Table \ref{tab:equalisol}. In the table we use the symbol $\star$ to denote the set consisting of an isolated closed point. 
We will consider the weights of the action with respect to a linearization on the very ample line bundle 
\begin{equation}\label{eq:minemb}
L:=L_i+\sum_{j\in J}L_j,
\end{equation} 
and denote by $Y_i$ the union of the fixed components of weight $i$. As usual, by choosing conveniently the linearization of the action to $L$, 
we may assume that $Y_0=F_-$, and then a case by case inspection shows then that the source of the action, which is a point, is attained at weight $3$. 

\begin{table}[!h!!]\caption{Fixed point components of the $\C^*$-action on $X=\ol{X}(J)$.
\label{tab:inducedF2}}
\begin{center}
\begin{tabular}{|c|c|c|c|c|c|}
\hline
$\cD$&$\cD(i)$&$Y_0=\cD_i(J)$&$Y_1$&$Y_2$&$Y_3$\\\hline\hline
\multirow{3}{*}{$\DA_n$}&\multirow{3}{*}{$\DA_n(i)$}&$\DA_{i-1}(1)$&$\DA_{i-2}(1)$&$\DA_{i-2}(1)$&\multirow{3}{*}{$\star$}\\
&&$\times$&$\times$&$\sqcup$&\\
&&$\DA_{n-i}(n-i)$&$\DA_{n-i-1}(n-i-1)$&$\DA_{n-i-1}(n-i-1)$&\\\hline
$\DC_n$&$\DC_n(n)$&$\DA_{n-1}(1)$&$\DA_{n-2}(1)$&$\emptyset$&$\star$\\\hline
$\DB_n$&$\DB_n(1)$&$\DB_{n-1}(1)$&$\star$&$\DB_{n-2}(1)$&$\star$\\\hline
$\DD_n$&$\DD_n(1)$&$\DD_{n-1}(1)$&$\star$&$\DD_{n-2}(1)$&$\star$\\\hline
$\DD_n $&$\DD_n(n)$&$\DA_{n-1}(2)$&$\DA_{n-3}(2)$&$\DA_{1}(1)\times \DA_{n-3}(1)$&$\star$\\\hline
$\DE_6$&$\DE_6(1)$&$\DD_5(4)$&$\DA_4(1)$&$\DA_4(3)$&$\star$\\\hline
$\DE_7$&$\DE_7(7)$&$\DE_6(1)$&$\DD_5(1)$&$\DD_5(5)$&$\star$\\\hline
\end{tabular}
\end{center}
\end{table}

One may easily compute the ranks of the positive and negative parts of the normal bundles of the inner fixed point components of the action, which allow to compute the exceptional locus of $\psi$ and its inverse. As usual, let us denote by $X^\flat\to X$ the blowup of $X$ along its source $Y_3$, with exceptional divisor $\P^{\dim(\cD_i(J))}$; note that the sink $Y_0=F_-$ of $X$ is already a divisor.

\begin{lemma}\label{lem:FHexc}
With the above notation, 
$\nu^+(Y_1)=1$, and $\nu^-(Y_2)=1$.  
In particular, $\Exc(\psi)=\ol{X^+(Y_1)}\cap Y_0\simeq Y_1$ and $\Exc(\psi^{-1})=\ol{{X^\flat}^-(Y_2)}\cap \P^{\dim(\cD_i(J))}\simeq Y_2$.
\end{lemma}

\begin{proof}
First of all, we note that $\ol{X^-(Y_1)}=F_+$ is a divisor in $X$, hence, by Formula (\ref{eq:nus}), $\nu^+(Y_1)=1$. On the other hand, the fact that the source of $X$ is a point, easily implies 
that $\nu^-(Y_2)=1$ (in the cases in which $Y_2\neq\emptyset$). In fact, if this were not the case, we would have a positive dimensional family of curves joining a given point of $Y_2$ and the point $Y_3$. This contradicts the fact that, by \ref{sssec:AMvsFM}, these curves are lines in the embedding given by the very ample line bundle $L$.
By Lemma \ref{lem:defincodim1}, this implies that 
$$\Exc(\psi)\subseteq\ol{{X^\flat}^+(Y_1)}\cap Y_0\simeq Y_1,\quad \Exc(\psi^{-1})\subseteq\ol{{X^\flat}^-(Y_2)}\cap \P^{\dim(\cD_i(J))}\simeq Y_2.$$
Note that, since $Y_-\subset X$ is a divisor, ${X^\flat}^+(Y_1)\simeq X^+(Y_1)$,

Note also that in the case in which $\cD=\DC_n$ this already says that $\psi^{-1}$ is a morphism. On the other hand the variety $F_+$ is a projective space together with the equalized action having as fixed components a hyperplane $Y_1$ and a point $Y_3$; in particular we see that also $\nu^-(Y_1)=1$ and so also $\psi$ is an isomorphism. 

In the rest of the cases $\psi^{-1}$ contracts the divisor $\ol{X^{\flat-}(Y_1)}\cap \P^{\dim(\cD_i(J))}$ to $\ol{X^-(Y_1)}\cap F_-\subset F_-$, so that this set, which is isomorphic to $Y_1$, is necessarily contained in $\Exc(\psi)$. A similar arguments determines $\Exc(\psi^{-1})$.
\end{proof}

\begin{remark}
As we have just seen, in the case of $\DC_n(n)$, $\psi^{-1}$ and $\psi$ are isomorphisms; in the rest of the cases we obtain examples of non-regular birational transformations. The cases in which $\Pic(Y_0)=1$ are precisely the homogeneous examples in the list of special birational transformations of type $(2,1)$ classified by Fu and Hwang, \cite{FuHw}. In a nutshell, the maps $\psi$ are  birational linear projections of the minimal embeddings of the rational homogeneous spaces $\cD_i(J)$ listed above. The exceptional loci of $\psi$ and its inverse can be computed by using Lemma \ref{lem:FHexc} in each case (see Table \ref{tab:FuHw} below).

\begin{table}[!h!!]\caption{(Homogeneous) Special birational transformations of type $(2,1)$.
\label{tab:FuHw}}
\begin{center}
\begin{tabular}{|c|c|c|c|}
\hline
\multicolumn{2}{|c|}{$\hspace{0.43cm}\cD_i(J)\hspace{0.1cm}\stackrel{\psi}{\dashrightarrow}\hspace{0.15cm}\P^{\dim(\cD_i(J))}$}&\multirow{2}{*}{$\Exc(\psi)$}&\multirow{2}{*}{$\Exc(\psi^{-1})$}\\\cline{1-2}
$\cD_i(J)$&$\dim(\cD_i(J))$&&\\\hline\hline
$\DA_{i-1}(1)$&$\multirow{3}{*}{$n-1$}$&$\DA_{i-2}(1)$&$\DA_{i-2}(1)$\\
$\times$&&$\times$&$\sqcup$\\
$\DA_{n-i}(n-i)$&&$\DA_{n-i-1}(n-i-1)$&$\DA_{n-i-1}(n-i-1)$\\\hline
$\DB_{n-1}(1)$&$2n-3$&$\star$&$\DB_{n-2}(1)$\\\hline
$\DD_{n-1}(1)$&$2n-4$&$\star$&$\DD_{n-2}(1)$\\\hline
$\DA_{n-1}(2)$&$2(n-2)$&$\DA_{n-3}(2)$&$\DA_{1}(1)\times \DA_{n-3}(1)$\\\hline
$\DD_5(4)$&$10$&$\DA_4(1)$&$\DA_4(3)$\\\hline
$\DE_6(1)$&$16$&$\DD_5(1)$&$\DD_5(5)$\\\hline
\end{tabular}
\end{center}
\end{table}

\end{remark}

\subsection{Linear sections}\label{ssec:linsec}

There are still two examples in the list of \cite{FuHw} that do not appear in Table \ref{tab:FuHw}. We will show now how to obtain some geometric realizations of them. These realizations will be $\C^*$-invariant subvarieties of the varieties $X=\ol{X}(J)$ described above. 

We start by considering the projection $X\to\ell$, and noting that this is a $\cD_i(J)$-bundle, which, by Grothendieck's theorem, is determined by a cocycle $\theta\in \HH^1(\ell,G^\perp)$, 
which is completely determined by a co-character $\sigma':\Mo(H^\perp)\to \Z$. One may then easily check (see \cite[Example~2.3]{MOS6}) that this cocycle is equal to the co-character $\sigma_+$ defining the restriction of the $\C^*$-action on $\ol{X}(J)$ to the fiber $F_+$ (see Lemma \ref{lem:actionF-}). The cocycle $\theta$ defines then a principal $\C^*$-bundle $\cE\to \ell$, so that $X$ may be written as the variety $\cE\times^{\C^*}\cD_i(J)$, which is defined as the quotient of $\cE\times^{\C^*}\cD_i(J)$ by the equivalence relation:
$$
(e,x)\sim (et,tx)\mbox{, for every }t\in\C^*.
$$ 

At this point, for any $B\subset\cD_i(J)$ smooth $\C^*$-invariant subvariety, we may define a subvariety
$$
\ol{B}:=\cE\times^{\C^*}Y\subset \cE\times^{\C^*}\cD_i(J)=X,
$$
which is a $B$-bundle over $\P^1$, and inherits the action of $\C^*$. The fixed point locus of the $\C^*$-action on $\ol{B}$ will be $X^{\C^*}\cap \ol{B}$. In particular, if $B\subset \cD_i(J)$ contains the source of the $\C^*$-action defined by $\sigma^+$, $\ol{B}$ will then be a geometric realization of a birational map 
$$
B\dashrightarrow \P^{\dim(B)}.
$$
Furthermore, in order to avoid this map to be an isomorphism, we will require $B$ to meet both $Y_1,Y_2\subset X^{\C^*}$. 

The two non-homogeneous examples in the list of \cite[Theorem 1.1]{FuHw} can be achieved by considering complete intersections defined by general elements in the linear system $\HH^0(\cD_i(J),L)_1$. By definition they contain $Y_1$ and $Y_3$, and are smooth at its intersection with $Y_2$ by Bertini's theorem. Furthermore, such a section is smooth at $Y_3$, since $\P(T_{\cD_i(J),Y_3}^\vee)$ can be naturally identified with $\langle Y_2\rangle$, which $\langle B \rangle$ meets in dimension $\dim(B)-1$ , by definition. 

Consequentely, since $B$ supports a $\C^*$-action, in order to guarantee its smoothness it is enough to prove that it is smooth at $Y_1$. For this purpose we use \cite[Proposition 1.7.5]{BS} together with Proposition \ref{prop:smooth} and Remark \ref{rem:smooth}, to claim that a general $B$ of this kind is smooth if  
$\dim (Y_1)<\dim(\cD_i(J))/2$, or $\dim(Y_1)=\dim(\cD_i(J))/2$ and the top Chern class of $\cN_{Y_1|\cD_i(J)}\otimes L$ is zero. 
The following table contains the description of the normal bundles of $Y_1$ in $\cD_i(J)$ in each case:
\begin{table}[!h!!]\caption{The normal bundle $\cN:=\cN_{Y_1|\cD_i(J)}$.
\label{tab:FuHw2}}
\begin{tabular}{|C||C|C|C|C|C|}
\hline
\cD(i)&\DA_n(i)&\DC_n(n)&\DD_n(n)&\DE_6(1)&\DE_7(7)\\\hline
\cD_i(J)&\DA_{i-1}(1)\times \DA_{n-i}(n-i)&\DA_{n-1}(1)&\DA_{n-1}(2)&\DD_5(4)&\DE_6(1)\\\hline
Y_1&\DA_{i-2}(1)\times \DA_{n-i-1}(n-i-1)&\DA_{n-2}(1)&\DA_{n-3}(2)&\DA_4(1)&\DD_5(1)\\\hline
\cN&\cO(1,0)\oplus\cO(0,1)&\cO(1)&\cQ^{\oplus 2}&\bigwedge^2T(-2)&\cS\\\hline
\end{tabular}
\end{table}

This shows that this construction works, and provides examples different from the ones we have already described  in the following two cases, which are the ones included in \cite[Theorem 1.1]{FuHw}:
\begin{itemize}
\item[(i)] $\cD_i(J)=\DA_4(2)$, and $B\subset \cD_i(J)$ is a general linear complete intersection of codimension at most two.
\item[(ii)] $\cD_i(J)=\DD_5(4)$, and $B\subset \cD_i(J)$ is a general linear complete intersection of codimension at most three. 
\end{itemize}

\begin{remark}\label{rem:tango}
The existence of a general linear complete intersection of codimension at most three of (ii) is guaranteed by the existence of an exact sequence of vector bundles in $\P^4$:
$$
\shse{\cO_{\P^4}}{\bigwedge^2\Omega_{\P^4}(3)}{\cT}
$$
where $\cT$ denotes the so-called Tango bundle in $\P^4$ (cf. \cite{JSW}).
\end{remark}